\documentclass{amsart}
\usepackage{graphicx}
\newtheorem{thm}{Theorem}[section]

\newtheorem{lem}[thm]{Lemma}
\newtheorem{prop}[thm]{Proposition}
\theoremstyle{definition}

\theoremstyle{remark}
\theoremstyle{Proof}

\numberwithin{equation}{section}
\author[O. Ajebbar]{ Ajebbar Omar}
\address{ Ajebbar Omar \\Department of Mathematics\\
Ibn Zohr University, Faculty of Sciences, Agadir\\
Morocco.} \email{omar-ajb@hotmail.com}
\author[E. Elqorachi]{  Elqorachi Elhoucien}
\address{Elqorachi Elhoucien \\Department of Mathematics\\
Ibn Zohr University, Faculty of Sciences, Agadir\\
Morocco.} \email{elqorachi@hotmail.com}
\author[Th. M. Rrassias]{Themistocles M. Rassias}
\address{ Themistocles M. Rassias,\\Department of Mathematics,\\ National Technical University of Athens\\Zofrafou Campus, 15780 Athens, Greece. }
\email{ trassias@math.ntua.gr}

\begin{document}

\title[The stability of a Cosine-Sine Functional
equation]{The stability of a Cosine-Sine Functional Equation on
abelian groups}

\keywords{Hyers-Ulam stability; Group; Amenable group; Cosine
equation; Sine equation; Additive function; Multiplicative
function.}
\thanks{2010 Mathematics Subject
Classification. Primary 39B82; Secondary 39B32}
%\submitted{May 04, 2018}%
%\date{}%
%\dedicatory{}%
%\commby{}%
% ----------------------------------------------------------------
\begin{abstract}
In this paper we establish the stability of the functional equation
\begin{equation*}f(x-y)=f(x)g(y)+g(x)f(y)+h(x)h(y)),\;\; x,y \in G, \end{equation*}where $G$ is an abelian group.
\end{abstract}
\maketitle
% ----------------------------------------------------------------
\section{Introduction}
In many studies concerning functional equations related to the
Cauchy equation $f(xy)=f(x)f(y)$, the main tool is a kind of
stability problem inspired by the famous problem proposed in 1940 by
Ulam \cite{Ulam}.\\More precisely, given a group $G$ and a metric
group $H$ with metric $d$, it is asked if for every function
$f:G\rightarrow H$, such that the function $(x,y)\mapsto
f(xy)-f(x)f(y)$ is bounded, there exists a homomorphism
$\chi:G\rightarrow H$ such that the function $x\mapsto
d(f(x),\chi(x))$ is bounded.\\The first affirmative answer to Ulam's
question was given in 1941 by Hyers \cite{Hyers}, under the
assumption that $G$ and $H$ are Banach spaces.\\After Hyers's result
a great number of papers on the subject have been published,
generalizing Ulam's problem and Hyers's result in various
directions. The interested reader should refer to \cite{Brzdek,
Caradiu, Gavruta, Moslehian, M.Rassias, Rassias} for a thorough
account on the subject of the stability of functional equations.\\In
this paper we will investigate the stability problem for the
trigonometric functional equation
\begin{equation}\label{eq1}f(x-y)=f(x)g(y)+g(x)f(y)+h(x)h(y),\,x,y\in G\end{equation}
on abelian groups.\\In \cite{Székelyhidi} Sz\'{e}kelyhidi proved the
Hyers-Ulam stability for the functional equation
$$f(xy)=f(x)g(y)+g(x)f(y),\,x,y\in G$$ and cosine
functional equation $$g(xy)=g(x)g(y)-f(x)f(y),\,x,y\in G$$ on
amenable group $G$. Chung, choi and Kim \cite{Chung} poved the
Hyers-Ulam stability of $$f(x+\sigma(y))=f(x)g(y)-g(x)f(y),\,x,y\in
G$$ where
$\sigma:G\rightarrow G$ is an involution of the abelian group $G$.\\
Recently, in \cite{Ajebbar and Elqorachi2,Ajebbar and Elqorachi3}
the authors obtained the stability of the functional equations
$$f(xy)=f(x)g(y)+g(x)f(y)+h(x)h(y),\,x,y\in G,$$
$$f(x\sigma(y))=f(x)g(y)+g(x)f(y),\,x,y\in G,$$
$$f(x\sigma(y))=f(x)f(y)-g(x)g(y),\,x,y\in G$$
and $$f(x\sigma(y))=f(x)g(y)-g(x)f(y),\,x,y\in G$$ on amenable
groups, and where $\sigma:G\rightarrow G$ is an involutive
automorphism.\\The aim of the present paper is to extend the
previous  results to the functional equation (\ref{eq1}) on abelian
groups.
\section{Definitions and Notations}
Throughout this paper $(G,+)$ denotes an abelian group with the
identity element $e$. We denote by $\mathcal{B}(G)$ the linear space
of all bounded complex-valued
functions on $G$.\\
Let $\mathcal{V}$ be a linear space of complex-valued functions on
$G$. We say that the functions
$f_{1},\cdot\cdot\cdot,f_{n}:G\rightarrow\mathbb{C}$ are linearly
independent modulo $\mathcal{V}$ if
$\lambda\,_{1}\,f_{1}+\cdot\cdot\cdot+\lambda\,_{n}\,f_{n}\in
\mathcal{V}$ implies that
$\lambda\,_{1}=\cdot\cdot\cdot=\lambda\,_{n}=0$ for any
$\lambda\,_{1},\cdot\cdot\cdot,\lambda\,_{n}\in \mathbb{C}$. We say
that the linear space $\mathcal{V}$ is two-sided invariant if $f\in
\mathcal{V}$ implies that the function $x\mapsto f(x+y)$
belongs to $\mathcal{V}$ for any $y\in G$.\\
If $I$ is the identity map of $G$ we say that $\mathcal{V}$ is
$(-I)$-invariant if $f\in\mathcal{V}$ implies that the function
$x\mapsto f(-x)$ belongs to $\mathcal{V}$.\\The space
$\mathcal{B}(G)$ is an obvious example of a linear space of
complex-valued functions on $G$ which is two-sided invariant and
$(-I)$-invariant.\\Let $f:G\rightarrow\mathbb{C}$ be a function. We
denote respectively by $f^{e}(x):=\frac{f(x)+f(-x)}{2},\,x\in G$ and
$f^{o}(x):=\frac{f(x)-f(-x)}{2},\,x\in G$ the even part and the odd
part of $f$.
\section{Basic results}
In this section we present some general stability properties of the
functional equation (\ref{eq1}).\\
Throughout this section we let $\mathcal{V}$ denote a two-sided
invariant and $(-I)$-invariant linear space of complex-valued
functions on $G$.
\begin{lem}Let $f,g,h:G\rightarrow\mathbb{C}$ be functions. Suppose that the functions
\begin{equation*}x\mapsto f(x-y)-f(x)g(y)-g(x)f(y)-h(x)h(y)\end{equation*}
and
\begin{equation}\label{eq2}x\mapsto f(x-y)-f(y-x)\end{equation}
belong to $\mathcal{V}$ for all $y\in G$, then\\
(1) $f^{o}\in\mathcal{V}$.\\
(2) The following functions
$\varphi_{1},\varphi_{2},\varphi_{3}:G\times G\rightarrow\mathbb{C}$
\begin{equation}\label{eq3}
f^{e}(x)g^{o}(y)+g^{e}(x)f^{o}(y)+h^{e}(x)h^{o}(y)=\varphi_{1}(x,y),\end{equation}
\begin{equation}\label{eq4}g^{o}(x)f^{e}(y)+h^{o}(x)h^{e}(y)=\varphi_{2}(x,y),\end{equation}
\begin{equation}\label{eq5}
f(x+y)-f(x-y)+2f^{o}(x)g^{o}(y)+2g^{o}(x)f^{o}(y)+2h^{o}(x)h^{o}(y)=\varphi_{3}(x,y)\end{equation}
are such that the functions $x\mapsto\varphi_{1}(x,y)$,
$x\mapsto\varphi_{1}(y,x)$, $x\mapsto\varphi_{2}(x,y)$,
$x\mapsto\varphi_{3}(x,y)$ and $x\mapsto\varphi_{3}(y,x)$ belong to
$\mathcal{V}$ for all $y\in G$.
\end{lem}
\begin{proof}
By putting $y=e$ in (\ref{eq2}) we get that the function $x\mapsto
f(x)-f(-x)$ belongs to $\mathcal{V}$ which proves $(1)$.\\
Let $\psi$ be the function defined on $G\times G$ by
\begin{equation}\label{eq6}\psi(x,y)=f(x-y)-f(x)g(y)-g(x)f(y)-h(x)h(y).\end{equation}
From (\ref{eq6}) we can verify easily that
\begin{equation}\label{eq7}\psi(x,y)=f^{e}(x-y)+f^{o}(x-y)-f^{e}(x)g(y)-f^{o}(x)g(y)-g(x)f(y)-h(x)h(y).\end{equation}
Now let
\begin{equation}\label{eq8}\phi(x,y):=\psi(x,y)-f^{o}(x-y)+f^{o}(x)g(y).\end{equation}
Then by using (\ref{eq7}) and (\ref{eq8}) we get
\begin{equation}\label{eq9}f^{e}(x-y)=f^{e}(x)g(y)+g(x)f(y)+h(x)h(y)+\phi(x,y).\end{equation}
Since $f^{e}$ is an even function on the abelian group $G$ then we
have
$$f^{e}(x-y)=f^{e}(-(x-y))=f^{e}((-x)-(-y)).$$ Hence, by
applying (\ref{eq9}) to the pair $(-x,-y)$, we obtain
\begin{equation}\label{eq10}f^{e}(x-y)=f^{e}(x)g(-y)+g(-x)f(-y)+h(-x)h(-y)+\phi(-x,-y).\end{equation} When we
subtract equation (\ref{eq10}) from (\ref{eq9}) we get that
\begin{equation}\label{eq11}\begin{split}&2f^{e}(x)g^{o}(y)+g(x)f(y)-g(-x)f(-y)+h(x)h(y)-h(-x)h(-y)\\
&=\phi(-x,-y)-\phi(x,y).\end{split}\end{equation} For the pair
$(-x,y)$ the identity (\ref{eq11}) becomes
\begin{equation}\label{eq12}\begin{split}&2f^{e}(x)g^{o}(y)+g(-x)f(y)-g(x)f(-y)+h(-x)h(y)-h(x)h(-y)\\
&=\phi(x,-y)-\phi(-x,y).\end{split}\end{equation} By adding
(\ref{eq11}) and (\ref{eq12}) we obtain
\begin{equation*}\begin{split}&4f^{e}(x)g^{o}(y)+2g^{e}(x)[f(y)-f(-y)]+2h^{e}(x)[h(y)-h(-y)]\\
&=\phi(-x,-y)-\phi(x,y)+\phi(x,-y)-\phi(-x,y).\end{split}\end{equation*}
Hence the identity (\ref{eq3}) can be written as follows
where\begin{equation*}\varphi_{1}(x,y):=\frac{1}{4}[\phi(-x,-y)-\phi(x,y)+\phi(x,-y)-\phi(-x,y)].\end{equation*}
By using (\ref{eq8}) and the identity above we get, by an elementary
computation, that
\begin{equation}\label{eq13}\varphi_{1}(x,y)=\frac{1}{4}[\psi(-x,-y)-\psi(x,y)+\psi(x,-y)-\psi(-x,y)+2f^{o}(x-y)-2f^{o}(x+y)].\end{equation}
By interchanging  $x$ and $y$ in (\ref{eq3}) we obtain
$$g^{o}(x)f^{e}(y)+f^{o}(x)g^{e}(y)+h^{o}(x)h^{e}(y)=\varphi_{1}(y,x),$$
and then we get
\begin{equation}\label{eq14}\varphi_{2}(x,y):=\varphi_{1}(y,x)-f^{o}(x)g^{e}(y),\,x,y\in G.\end{equation}
On the other hand, by replacing $y$ by $-y$ in (\ref{eq6}) we get
that
\begin{equation}\label{eq15}\psi(x,-y)=f(x+y)-f(x)g(-y)-g(x)f(-y)-h(x)h(-y).\end{equation}
By subtracting the result of equation (\ref{eq6}) from the result of
equation (\ref{eq15}) we obtain
\begin{equation*}\begin{split}&f(x+y)-f(x-y))=-2f(x)g^{o}(y)-2g(x)f^{o}(y)-2h(x)h^{o}(y)\\&+\psi(x,-y)-\psi(x,y)\\
&=-2f^{e}(x)g^{o}(y)-2f^{o}(x)g^{o}(y)-2g^{e}(x)f^{o}(y)-2g^{o}(x)f^{o}(y)\\
&-2h^{e}(x)h^{o}(y)-2h^{o}(x)h^{o}(y)+\psi(x,-y)-\psi(x,y)\\
&=-2f^{o}(x)g^{o}(y)-2g^{o}(x)f^{o}(y)-2h^{o}(x)h^{o}(y)\\
&-2[f^{e}(x)g^{o}(y)+g^{e}(x)f^{o}(y)+h^{e}(x)h^{o}(y)]
+\psi(x,-y)-\psi(x,y)\\
&=-2f^{o}(x)g^{o}(y)-2g^{o}(x)f^{o}(y)-2h^{o}(x)h^{o}(y)-2\varphi_{1}(x,y)+\psi(x,-y)-\psi(x,y).\end{split}\end{equation*}
So that the identity (\ref{eq5}) can be written as follows
\begin{equation}\label{eq16}
\varphi_{3}(x,y):=-2\varphi_{1}(x,y)+\psi(x,-y)-\psi(x,y).\end{equation}
Sine $x$ and $y$ are arbitrary, by using the fact that the functions
$x\mapsto\psi(x,y)$, $x\mapsto f(x-y)-f(y-x)$ and $f^{o}$ belong to
the two-sided invariant and $(-I)$-invariant linear space
$\mathcal{V}$ of complex-valued functions on $G$  for all $y\in G$,
and taking (\ref{eq13}), (\ref{eq14}) and (\ref{eq16}) into account,
we deduce the rest of the proof.
\end{proof}
\begin{lem}Let $f,g,h:G\rightarrow\mathbb{C}$
be functions. Suppose that $f$ and $h$ are linearly independent
modulo $\mathcal{V}$, and that $h^{o}\not\in\mathcal{V}$. If the
functions
\begin{equation*}x\mapsto f(x-y)-f(x)g(y)-g(x)f(y)-h(x)h(y)\end{equation*}
and
\begin{equation*}x\mapsto f(x-y)-f(y-x)\end{equation*}
belong to $\mathcal{V}$ for all $y\in G$, then\\
(1)
\begin{equation}\label{eq17}h^{e}=\gamma\,f^{e}\end{equation} and
\begin{equation}\label{eq18}g^{o}=-\gamma\,h^{o}-\eta\,f^{o},\end{equation}
where $\gamma\,,\eta\in\mathbb{C}$ are constants.\\
(2) Moreover, if $f^{o}\neq0$, then
\begin{equation}\label{eq19}g^{e}=\eta\, f^{e}+\varphi,\end{equation}
where $\varphi\in\mathcal{V}$ and $\varphi(-x)=\varphi(x)$ for all
$x\in G$.
\end{lem}
\begin{proof}Since $f$ and $h$ are linearly independent modulo $\mathcal{V}$ then $f\not\in\mathcal{V}$.
According to Lemma 3.1(1) we have $f^{o}\in\mathcal{V}$, then
$f^{e}\not\in\mathcal{V}$ and consequently $f^{e}\neq0$.  So, there
exists $y_{0}\in G$ such that $f^{e}(y_{0})\neq0$. By putting
$y=y_{0}$ in (\ref{eq4}) we derive that there exist a constant
$\gamma\,\in\mathbb{C}$ and a function $b_{1}\in\mathcal{V}$ such
that
\begin{equation}\label{eq20}g^{o}=-\gamma\, h^{o}+b_{1}.\end{equation}
When we substitute this in (\ref{eq4}) we obtain $(-\gamma\,
h^{o}(x)+b_{1}(x))f^{e}(y)+h^{o}(x)h^{e}(y)=\varphi_{2}(x,y)$ ,
which implies $(h^{e}(y)-\gamma\,
f^{e}(y))h^{o}(x)=\varphi_{2}(x,y)-f^{e}(y)b_{1}(x)$. So, $x$ and
$y$ being arbitrary, we deduce that the function $x\mapsto
(h^{e}(y)-\gamma\, f^{e}(y))h^{o}(x)$ belongs to $\mathcal{V}$ for
all $y\in G$. As $h^{o}\not\in\mathcal{V}$ we get (\ref{eq17}).\\On
the other hand we get, from (\ref{eq3}), (\ref{eq17}) and
(\ref{eq20}), that
\begin{equation}\label{eq21}\varphi_{1}(x,y)=f^{e}(x)b_{1}(y)+g^{e}(x)f^{o}(y)\end{equation}for all
$x,y\in G$.
\par If $f^{o}\neq0$ then from (\ref{eq21})
there exist a constant $\eta\in\mathbb{C}$ and a function
$\varphi\in\mathcal{V}$ such that $g^{e}=\eta\, f^{e}+\varphi$ and
$\varphi(-x)=\varphi(x)$ for all $x\in G$. This is the result (2) of
Lemma 3.2. When we substitute this in the identity (\ref{eq21}) we
get, by a small computation, that
$\varphi_{1}(x,y)=f^{e}(x)[b_{1}(y)+\eta\,f^{o}(y)]+\varphi(x)f^{o}(y)$
for all $x,y\in G$. As the functions $\varphi$ and
$x\mapsto\varphi_{1}(x,y)$ belong to $\mathcal{V}$ for all $y\in G$,
we deduce that the function $x\mapsto
f^{e}(x)[b_{1}(y)+\eta\,f^{o}(y)]$ belongs to $\mathcal{V}$ for all
$y\in G$. So, taking into account that $f^{e}\not\in\mathcal{V}$ we
infer that $b_{1}=-\eta\, f^{o}$.
\par If $f^{o}=0$ then we get from (\ref{eq21}), and seeing
that $f^{e}\not\in\mathcal{V}$, that $b_{1}=0$.\\Hence, in both
cases we have $b_{1}=-\eta\, f^{o}$. By substituting this back into
(\ref{eq20}) we obtain (\ref{eq18}). This completes the proof.
\end{proof}
\begin{prop}Let $m:G\rightarrow\mathbb{C}$ be a nonzero multiplicative function
such that $m(-x)=m(x)$ for all $x\in G$. The solutions
$f,h:G\rightarrow\mathbb{C}$ of the functional equation
\begin{equation}\label{eq22}f(x+y)=f(x)m(y)+m(x)f(y)+h(x)h(y),\,\,x,y\in G\end{equation}
such that $f(-x)=f(x)$, $h(-x)=-h(x)$ for all $x\in G$ and $h\neq0$ are the pairs\\
$$f=\frac{1}{2}a^{2}m\,\,\,\textit{and}\,\,\,h=a\,m,$$ where
$a:G\rightarrow\mathbb{C}$ is a nonzero additive function.
\end{prop}
\begin{proof} It is easy to check that the indicated functions are solutions.
It is thus left to show that any solutions
$f,h:G\rightarrow\mathbb{C}$ can be written in the indicated forms.
Replacing $y$ by $-y$ in (\ref{eq22}) yields to the functional
equation $f(x-y)=f(x)m(y)+m(x)f(y)-h(x)h(y)$, because $f$ and $m$
are even functions, and $h$ is an add one. When to this we add
(\ref{eq22}) we get that $f(x+y)+f(x-y)=2f(x)m(y)+2m(x)f(y)$. Notice
that $m(x)\neq0$ for all $x\in G$, because $m$ is a nonzero
multiplicative function on the group $G$. Moreover since
$m(-x)=-m(x)$ for all $x\in G$ we have $m(x+y)=m(x-y)=m(x)m(y)$ for
all $x,y\in G$. So, by dividing both sides of (\ref{eq22}) by
$m(x+y)$ we get that $F:=f/m$ satisfies the classical quadratic
functional equation $F(x+y)+F(x-y)=2F(x)+2F(y)$. Hence from
\cite[Theorem 13.13]{Stetkaer} we derive that $F$ has the form
$F(x)=Q(x,x)$, $x\in G$, where $Q:G\times G\rightarrow\mathbb{C}$ is
a symmetric, bi-additive map. So that
\begin{equation}\label{eq23}f(x)=Q(x,x)m(x)\end{equation} for all $x\in
G$. Substituting this in (\ref{eq22}) and dividing both sides by
$m(x+y)=m(x)m(y)$, and using that $Q$ is a symmetric, bi-additive
map we derive that
\begin{equation}\label{eq24}2Q(x,y)=H(x)H(y)\end{equation} for all $x,y\in G$ with $H:=h/m$. Since, $H$ is a nonzero
function on $G$, because $h$ is, we get that there exists $y_{0}\in
G$ such that $H(y_{0})\neq0$. Hence, by putting $y=y_{0}$ in the
last identity and dividing both sides by $H(y_{0})$, and taking into
account that $Q$ is bi-additive, we deduce that $H=a$, where
$a:G\rightarrow\mathbb{C}$ is additive. So $h=a\,m$. Notice that $a$
is nonzero. On the other hand, by replacing $H$ by $a$ in
(\ref{eq24}) and putting $x=y$ we deduce that
$Q(x,x)=\frac{1}{2}a^{2}(x)$ for all $x\in G$. When we substitute
this in (\ref{eq23}) we get that $f=\frac{1}{2}a^{2}m$. This
completes the proof.
\end{proof}
\begin{prop} Let $f,g,h:G\rightarrow\mathbb{C}$
be functions. Suppose that $f$ and $h$ are linearly independent
modulo $\mathcal{B}(G)$. If the function
\begin{equation*}(x,y)\mapsto f(x+y)-f(x)g(y)-g(x)f(y)-h(x)h(y)\end{equation*}
is bounded then we have one of the
following possibilities:\\
(1)\[ \left\{
\begin{array}{r c l}
f&=&-\lambda^{2}f_{0}+\lambda^{2}b,\quad\quad\quad\quad\quad\quad\quad\quad\quad\quad\quad\quad\quad\quad\quad\quad\quad\quad\quad\quad\quad\quad\\
g&=&\frac{1+\rho^{2}}{2}f_{0}+\rho\, g_{0}+\frac{1-\rho^{2}}{2}b,\quad\quad\quad\quad\quad\quad\quad\quad\quad\quad\quad\quad\quad\quad\quad\\
h&=&\lambda\,\rho\, f_{0}+\lambda\, g_{0}-\lambda\,\rho\,
b,\quad\quad\quad\quad\quad\quad\quad\quad\quad\quad\quad\quad\quad\quad\quad\quad\quad\quad\quad
\end{array}
\right.
\]
where $b:G\rightarrow\mathbb{C}$ is a bounded function,
$\rho\in\mathbb{C},\,\lambda\in\mathbb{C}\setminus\{0\}$ are
constants and $f_{0},g_{0}:G\rightarrow\mathbb{C}$ are functions
satisfying the cosine functional equation
$$f_{0}(x+y)=f_{0}(x)f_{0}(y)-g_{0}(x)g_{0}(y),\,\,x,y\in G;$$
(2)\[ \left\{
\begin{array}{r c l}
f&=&\lambda^{2}M+a\,m+b,\quad\quad\quad\quad\quad\quad\quad\quad\quad\quad\quad\quad\quad\quad\\
g&=&\beta\lambda\,(1-\frac{1}{2}\beta\lambda\,)M+(1-\beta\lambda\,)m-\frac{1}{2}\beta^{2}a\,m-\frac{1}{2}\beta^{2}b,\quad\quad\quad\quad\quad\quad\quad\\
h&=&\lambda\,(1-\beta\lambda\,)M-\lambda\, m-\beta a\,m-\beta
b,\quad\quad\quad\quad\quad\quad\quad\quad\quad\quad\quad\quad
\end{array}
\right.
\]
where $m:G\rightarrow\mathbb{C}$ is a nonzero bounded multiplicative
function, $M:G\rightarrow\mathbb{C}$ is a non bounded multiplicative
function, $a:G\rightarrow\mathbb{C}$ is a nonzero additive function,
$b:G\rightarrow\mathbb{C}$ is a bounded function and
$\beta\in\mathbb{C},\,\lambda\in\mathbb{C}\setminus\{0\}$ are constants;\\
(3)\[ \left\{
\begin{array}{r c l}
f&=&\frac{1}{2}a^{2}\,m+\frac{1}{2}a_{1}\,m+b,\quad\quad\quad\quad\quad\quad\quad\quad\quad\quad\quad\quad\quad\quad\quad\\
g&=&-\frac{1}{4}\beta^{2}a^{2}\,m+\beta a\,m-\frac{1}{4}\beta^{2} a_{1}\,m+m-\frac{1}{2}\beta^{2}b,\quad\quad\quad\quad\quad\quad\quad\quad\\
h&=&-\frac{1}{2}\beta a^{2}\,m+a\,m-\frac{1}{2}\beta
a_{1}\,m-\beta\,b,\quad\quad\quad\quad\quad\quad\quad\quad\quad\quad\quad\quad\quad
\end{array}
\right.
\]
where $m:G\rightarrow\mathbb{C}$ is a nonzero bounded multiplicative
function, $a,a_{1}:G\rightarrow\mathbb{C}$ are additive functions
such that $a$ is nonzero,
$b:G\rightarrow\mathbb{C}$ is a bounded function and $\beta\in\mathbb{C}$ is a constant;\\
(4) $$f(x+y)=f(x)m(y)+m(x)f(y)+(a(x)m(x)+b(x))(a(y)m(y)+b(y))$$for
all $x,y\in G$,
 $$g=-\frac{1}{2}\beta^{2}f+(1+\beta a)m+\beta b$$ and
$$h=-\beta f+a\,m+b,$$
where $m:G\rightarrow\mathbb{C}$ is a nonzero bounded multiplicative
function, $a:G\rightarrow\mathbb{C}$ is a nonzero additive function,
$b:G\rightarrow\mathbb{C}$ is a bounded function and
$\beta\in\mathbb{C}$ is a constant;\\
(5) $f(x+y)=f(x)g(y)+g(x)f(y)+h(x)h(y)$ for all $x,y\in G$.
\end{prop}
\begin{proof} We proceed as in the proof of \cite[Lemma 3.4]{Ajebbar and
Elqorachi3}.
\end{proof}
\section{Stability of equation (\ref{eq1})}
In this section we prove the main result of this paper.
\begin{thm} Let $f,g,h:G\rightarrow\mathbb{C}$ be functions. The
function
\begin{equation*}(x,y)\mapsto f(x-y)-f(x)g(y)-g(x)f(y)-h(x)h(y)\end{equation*}
is bounded if and only if one of the following assertions holds:\\
(1) $f=0$, $g$ is arbitrary and $h\in\mathcal{B}(G)$;\\
(2) $f,g,h\in\mathcal{B}(G)$;\\
(3)\[ \left\{
\begin{array}{r c l}
f&=&\alpha\,m-\alpha\,b,\quad\quad\quad\quad\quad\quad\quad\quad\quad\quad\quad\quad\quad\quad\quad\quad\quad\quad\quad\quad\quad\quad\quad\quad\\
g&=&\frac{1-\alpha\lambda^{2}}{2}m+\frac{1+\alpha\lambda^{2}}{2}b-\lambda\,\varphi,\quad\quad\quad\quad\quad\quad\quad\quad\quad\quad\quad\quad\quad\quad\quad\quad\\
h&=&\alpha\lambda\,m-\alpha\lambda\,b+\varphi,\quad\quad\quad\quad\quad\quad\quad\quad\quad\quad\quad\quad\quad\quad\quad\quad\quad\quad\quad\quad\quad
\end{array}
\right.
\]
where $m:G\rightarrow\mathbb{C}$ is a multiplicative function such
that $m(-x)=m(x)$ for all $x\in G$ or $m\in\mathcal{B}(G)$,
$b,\varphi:G\rightarrow\mathbb{C}$ are bounded functions and $\alpha\in\mathbb{C}\setminus\{0\},\lambda\in\mathbb{C}$ are constants;\\
(4)\[ \left\{
\begin{array}{r c l}
f&=&f_{0},\quad\quad\quad\quad\quad\quad\quad\quad\quad\quad\quad\quad\quad\quad\quad\quad\quad\quad\quad\quad\quad\quad\quad\quad\quad\quad\\
g&=&-\frac{\lambda^{2}}{2}f_{0}+g_{0}-\lambda\,b,\quad\quad\quad\quad\quad\quad\quad\quad\quad\quad\quad\quad\quad\quad\quad\quad\quad\quad\quad\\
h&=&\lambda\,
f_{0}+b,\quad\quad\quad\quad\quad\quad\quad\quad\quad\quad\quad\quad\quad\quad\quad\quad\quad\quad\quad\quad\quad\quad\quad
\end{array}
\right.
\]
where $b:G\rightarrow\mathbb{C}$ is a bounded function,
$\lambda\in\mathbb{C}$ is a constant and
$f_{0},g_{0}:G\rightarrow\mathbb{C}$ are functions satisfying the
functional equation
$$f_{0}(x-y)=f_{0}(x)g_{0}(y)+g_{0}(x)f_{0}(y),\,\,x,y\in G;$$
(5) \[ \left\{
\begin{array}{r c l}
f&=&-\lambda^{2}f_{0}+\lambda^{2}b,\quad\quad\quad\quad\quad\quad\quad\quad\quad\quad\quad\quad\quad\quad\quad\quad\quad\quad\quad\quad\quad\quad\\
g&=&\frac{1+\rho^{2}}{2}f_{0}+\rho\, g_{0}+\frac{1-\rho^{2}}{2}b,\quad\quad\quad\quad\quad\quad\quad\quad\quad\quad\quad\quad\quad\quad\quad\\
h&=&\lambda\,\rho\, f_{0}+\lambda\, g_{0}-\lambda\,\rho\,
b,\quad\quad\quad\quad\quad\quad\quad\quad\quad\quad\quad\quad\quad\quad\quad\quad\quad\quad\quad
\end{array}
\right.
\]
where $b:G\rightarrow\mathbb{C}$ is a bounded function,
$\rho\in\mathbb{C},\,\lambda\in\mathbb{C}\setminus\{0\}$ are
constants and $f_{0},g_{0}:G\rightarrow\mathbb{C}$ are functions
satisfying the cosine functional equation
$$f_{0}(x+y)=f_{0}(x)f_{0}(y)-g_{0}(x)g_{0}(y),\,\,x,y\in G,$$ such that
$f_{0}(-x)=f_{0}(x)$ and $g_{0}(-x)=g_{0}(x)$ for all $x\in G$;\\
(6)
\[\left\{
\begin{array}{r c l}
f&=&\lambda^{2}f_{0}-\lambda^{2}b,\quad\quad\quad\quad\quad\quad\quad\quad\quad\quad\quad\quad\quad\quad\quad\quad\quad\quad\quad\quad\quad\quad\\
g&=&\frac{1}{2}f_{o}+\frac{1}{2}b,\quad\quad\quad\quad\quad\quad\quad\quad\quad\quad\quad\quad\quad\quad\quad\\
h&=&\lambda\,
g_{0},\quad\quad\quad\quad\quad\quad\quad\quad\quad\quad\quad\quad\quad\quad\quad\quad\quad\quad\quad
\end{array}
\right.
\]
where $b:G\rightarrow\mathbb{C}$ is a bounded function,
$\lambda\in\mathbb{C}\setminus\{0\}$ is a constant and
$f_{0},g_{0}:G\rightarrow\mathbb{C}$ are functions satisfying the
cosine functional equation
$$f_{0}(x+y)=f_{0}(x)f_{0}(y)-g_{0}(x)g_{0}(y),\,\,x,y\in G,$$ such that
$f_{0}(-x)=f_{0}(x)$ and $g_{0}(-x)=-g_{0}(x)$ for all $x\in G$;\\
(7)\[ \left\{
\begin{array}{r c l}
f&=&\frac{1}{2}a^{2}\,m+b,\quad\quad\quad\quad\quad\quad\quad\quad\quad\quad\quad\quad\quad\quad\quad\quad\quad\quad\\
g&=&m,\quad\quad\quad\quad\quad\quad\quad\quad\quad\quad\quad\quad\quad\quad\quad\quad\quad\quad\quad\quad\quad\quad\quad\quad\\
h&=&-ia\,m,\quad\quad\quad\quad\quad\quad\quad\quad\quad\quad\quad\quad\quad\quad\quad\quad\quad\quad\quad\quad\quad\quad\quad\quad
\end{array}
\right.
\]
where $m:G\rightarrow\mathbb{C}$ is a nonzero bounded multiplicative
function, $a:G\rightarrow\mathbb{C}$ is a nonzero additive function
and $b:G\rightarrow\mathbb{C}$ is a bounded function such that $m(-x)=m(x)$ and $b(-x)=-b(x)$ for all $x\in G$;\\
(8) $f(x-y)=f(x)g(y)+g(x)f(y)+h(x)h(y)$ for all $x,y\in G$;\\
(9)\[\left\{
\begin{array}{r c l}
f&=&F_{0}+\varphi,\quad\quad\quad\quad\quad\quad\quad\quad\quad\quad\quad\quad\quad\quad\quad\quad\quad\quad\quad\quad\quad\quad\\
g&=&-\frac{1}{2}\delta^{2}F_{0}+G_{0}+\delta\,H_{0}-\rho\,\varphi,\quad\quad\quad\quad\quad\quad\quad\quad\quad\quad\quad\quad\quad\\
h&=&-\delta\,F_{0}+H_{0}-\delta\,\varphi,\quad\quad\quad\quad\quad\quad\quad\quad\quad\quad\quad\quad\quad\quad\quad\quad\quad\quad
\end{array}
\right.
\]
where $\rho\in\mathbb{C},\,\delta\in\mathbb{C}\setminus\{0\}$ are
constants and the functions
$F_{0},G_{0},H_{0}:G\rightarrow\mathbb{C}$ are of the forms (6)-(7)
under the same constraints, with $F_{0}(-x)=F_{0}(x)$,
$G_{0}(-x)=G_{0}(x)$, $H_{0}(-x)=-H_{0}(x)$,
$\varphi(-x)=-\varphi(x)$ for all $x\in G$, such that
\par (i) $b(-x)=b(x)$ for all $x\in G$ and
$\rho=\frac{1+\lambda^{}\delta^{2}}{2\lambda^{2}}$ if $F_{0}$,
$G_{0}$ and $H_{0}$ are of the form (6),
\par (ii) $b=0$ and $\rho=\frac{1}{2}\delta^{2}$ if $F_{0}$,
$G_{0}$ and $H_{0}$ are of the form (7).
\end{thm}
\begin{proof} To study the stability of the functional equation (\ref{eq1})
we will discuss two cases according to whether $f$ and $h$ are
linearly independent modulo $\mathcal{B}(G)$ or not.\\
\underline{Case A}: $f$ and $h$ are linearly dependent modulo
$\mathcal{B}(G)$. We split the discussion into the cases
$h\in\mathcal{B}(G)$ and $h\not\in\mathcal{B}(G)$.\\
\underline{Case A.1}: $h\in\mathcal{B}(G)$. Then the function
\begin{equation*}(x,y)\mapsto f(x-y)-f(x)g(y)-g(x)f(y)\end{equation*}
is bounded. Since the group $G$ is abelian it is an amenable group.
So, according to \cite[Theorem 3.3]{Ajebbar and Elqorachi2}, we have
of the following assertions:\\
(1) $f=0$, $g$ is arbitrary and $h\in\mathcal{B}(G)$. The result
occurs in (1) of Theorem 4.1.\\
(2) $f,g,h\in\mathcal{B}(G)$. The result
occurs in (2) of Theorem 4.1.\\
(3) $f=a\,m+b$ and $g=m$, where $a:G\rightarrow\mathbb{C}$ is an
additive function, $m:G\rightarrow\mathbb{C}$ is a bounded
multiplicative function and $b:G\rightarrow\mathbb{C}$ is a bounded
function such that $m(-x)=m(x)$ and $a(-x)=a(x)$ for all $x\in G$.
Then $2a(x)=a(x)+a(-x)=a(x-x)=a(e)=0$ for all $x\in G$. Hence $a(x)=0$ for all $x\in G$.
We deduce that $f,g,h\in\mathcal{B}(G)$. This is the result (2) of Theorem 4.1.\\
(4) $f=\alpha\,m-\alpha\,b$, $g=\frac{1}{2}\,m+\frac{1}{2}\,b$,
where $\alpha\in\mathbb{C}\setminus\{0\}$ is a constant,
$b:G\rightarrow\mathbb{C}$ is a bounded function and
$m:G\rightarrow\mathbb{C}$ is a multiplicative function such that
$m(-x)=m(x)$ for all $x\in G$ or $m\in\mathcal{B}(G)$. This is the result (3) of Theorem 4.1 for $\lambda\,=0$.\\
(5) $f(x-y))=f(x)g(y)+g(x)f(y)$ for all $x,y\in G$. So, taking into
account that $h\in\mathcal{B}(G)$, we obtain the result (4) of
Theorem 4.1 for $\lambda\,=0$.\\
\underline{Case A.2}: $h\not\in\mathcal{B}(G)$. Then
$f\not\in\mathcal{B}(G)$. Indeed if $f\in\mathcal{B}(G)$ then the
functions $x\mapsto f(x)g(y)$ and $x\mapsto f(x-y)$ belong to
$\mathcal{B}(G)$ for all $y\in G$. As the function
$x\mapsto\psi(x,y)$ belongs to $\mathcal{B}(G)$ for all $y\in G$ we
get that the function $x\mapsto g(x)f(y)+h(x)h(y)$ belongs to
$\mathcal{B}(G)$ for all $y\in G$. So, taking into account that
$h\not\in\mathcal{B}(G)$, we get that there exist a constant
$\alpha\in\mathbb{C}\setminus\{0\}$ and a function
$k\in\mathcal{B}(G)$ such that
\begin{equation}\label{eq25}h=\alpha g+k.\end{equation}
Substituting (\ref{eq25}) in (\ref{eq6}) we get, by an elementary
computation, that
$$\psi(x,y)=f(x-y)-k(x)k(y)-g(x)[f(y)+\alpha\,h(y)]-g(y)[f(x)+\alpha k(x)]$$
for all $x,y\in G$. It follows that the function $x\mapsto
g(x)[f(y)+\alpha h(y)]$ belongs to $\mathcal{B}(G)$ for all $y\in
G$, so that $h=-\frac{1}{\alpha} f$ or $g\in\mathcal{B}(G)$. Hence,
taking (\ref{eq25}) into account, we get that $h\in\mathcal{B}(G)$,
which contradicts the assumption on $h$. We deduce that
$f\not\in\mathcal{B}(G)$. Since $f$ and $h$ are linearly dependent
modulo $\mathcal{B}(G)$ we deduce that there exist a constant
$\lambda\in\mathbb{C}\setminus\{0\}$ and a function
$\varphi\in\mathcal{V}$ such that
\begin{equation}\label{eq26}h=\lambda\,f+\varphi.\end{equation}
When we substitute (\ref{eq26}) in (\ref{eq6}) we obtain by an
elementary computation
\begin{equation}\label{eq27}\psi(x,y)+\varphi(x)\varphi(y)=f(x-y)-f(x)\phi(y)-\phi(x)f(y)\end{equation}
for all $x,y\in G$, where
\begin{equation}\label{eq28}\phi:=g+\frac{\lambda^{2}}{2}f+\lambda\, \varphi.\end{equation}
Since the functions $\psi$ and $\varphi$ are bounded we derive from
(\ref{eq27}) that the function $(x,y)\mapsto
f(x-y)-f(x)\phi(y)-\phi(x)f(y)$ is also bounded. Hence, according to
\cite[Theorem 3.3]{Ajebbar and Elqorachi2} and taking (\ref{eq26})
into account and that $h\not\in\mathcal{B}(G)$,
we have one of the following possibilities:\\
(1) $f=a\,m+b$ and $\phi=m$, where $a:G\rightarrow\mathbb{C}$ is an
additive function, $m:G\rightarrow\mathbb{C}$ is a bounded
multiplicative function and $b:G\rightarrow\mathbb{C}$ is a bounded
function such that $m(-x)=m(x)$ and $a(-x)=a(x)$ for all $x\in G$.
As in Case A.1(3) we prove that the result (2) of Theorem 4.1 holds.\\
(2) $f=\alpha\,m-\alpha\,b$, $\phi=\frac{1}{2}m+\frac{1}{2}b$, where
$\alpha\in\mathbb{C}\setminus\{0\}$ is a constant,
$b:G\rightarrow\mathbb{C}$ is a bounded function and
$m:G\rightarrow\mathbb{C}$ is a multiplicative function such that
$m(-x)=m(x)$ for all $x\in G$ or $m\in\mathcal{B}(G)$. So, by using
(\ref{eq28}) and (\ref{eq26}) we get that
$g=\frac{1}{2}m+\frac{1}{2}b-\frac{\lambda^{2}}{2}(\alpha m-\alpha
b)-\lambda\, \varphi
=\frac{1-\alpha\lambda^{2}}{2}m+\frac{1+\alpha\lambda^{2}}{2}b-\lambda\,\varphi$
and $h=\alpha\lambda\, m-\alpha\lambda\, b+\varphi$. The result
occurs in (3) of Theorem 4.1.\\
(3) $f(x-y)=f(x)\phi(y)+\phi(x)f(y)$ for all $x,y\in G$. By putting
$f_{0}:=f$ and $g_{0}:=\phi$ we get the result (4) of Theorem 4.1.\\
\underline{Case B}: $f$ and $h$ are linearly independent modulo
$\mathcal{B}(G)$. Then $f\not\in\mathcal{B}(G)$. Moreover, according
to Lemma 3.1(1), we have $f^{o}\in\mathcal{B}(G)$ and then
$f^{e}\neq0$. It follows from (\ref{eq4}), with $\varphi_{2}$
satisfying the same constraint in Lemma 3.1, that if
$h^{o}\in\mathcal{B}(G)$ then $g^{o}\in\mathcal{\mathcal{B}}(G)$. So
we will discuss the following subcases: $h^{o}\in\mathcal{B}(G)$ and
$h^{o}\not\in\mathcal{B}(G)$.\\
\underline{Subcase B.1}: $h^{o}\in\mathcal{B}(G)$. Let $x,y\in G$ be
arbitrary. From (\ref{eq6}) we get, by using (\ref{eq3}) an
(\ref{eq4}), that
\begin{equation*}\begin{split}&f^{e}(x-y)=[f^{e}(x)+f^{o}(x)][g^{e}(y)+g^{o}(y)]+[g^{e}(x)+g^{o}(x)][f^{e}(y)+f^{o}(y)]\\
&=+[h^{e}(x)+h^{o}(x)][h^{e}(y)+h^{o}(y)]-f^{o}(x-y)+\psi(x,y)\\
&=f^{e}(x)g^{e}(y)+g^{e}(x)f^{e}(y)+h^{e}(x)h^{e}(y)+[f^{e}(x)g^{o}(y)+g^{e}(x)f^{o}(y)+h^{e}(x)h^{o}(y)]\\
&+[f^{o}(x)g^{e}(y)+g^{o}(x)f^{e}(y)+h^{o}(x)h^{e}(y)]+f^{o}(x)g^{o}(y)+g^{o}(x)f^{o}(y)+h^{o}(x)h^{o}(y)\\
&-f^{o}(x-y)+\psi(x,y)\\
&=f^{e}(x)g^{e}(y)+g^{e}(x)f^{e}(y)+h^{e}(x)h^{e}(y)+f^{o}(x)g^{o}(y)+g^{o}(x)f^{o}(y)+h^{o}(x)h^{o}(y)\\
&-f^{o}(x-y)+\varphi_{1}(x,y)+\varphi_{1}(y,x)+\psi(x,y).\end{split}\end{equation*}
So, $x$ and $y$ being arbitrary, by using the fact that the
functions $f^{o}$, $g^{o}$, $h^{o}$ and $\psi$ are bounded, and
taking (\ref{eq13}) into account, we deduce from the identity above
that the function $(x,y)\mapsto
f^{e}(x-y)-f^{e}(x)g^{e}(y)-g^{e}(x)f^{e}(y)-h^{e}(x)h^{e}(y)$ is
bounded, so is the function $(x,y)\mapsto
f^{e}(x+y)-f^{e}(x)g^{e}(y)-g^{e}(x)f^{e}(y)-h^{e}(x)h^{e}(y)$.
Moreover since the functions $f$ and $h$ are linearly independent
modulo $\mathcal{B}(G)$ and $f^{o},h^{o}\in\mathcal{B}(G)$ we get
that $f^{e}$ and $h^{e}$ are linearly independent. Hence, according
to Proposition 3.4 we have one of the following
possibilities:\\
(1)\[ \left\{
\begin{array}{r c l}
f^{e}&=&-\lambda^{2}f_{0}+\lambda^{2}b,\quad\quad\quad\quad\quad\quad\quad\quad\quad\quad\quad\quad\quad\quad\quad\quad\quad\quad\quad\quad\quad\quad\\
g^{e}&=&\frac{1+\rho^{2}}{2}f_{0}+\rho\, g_{0}+\frac{1-\rho^{2}}{2}b,\quad\quad\quad\quad\quad\quad\quad\quad\quad\quad\quad\quad\quad\quad\quad\\
h^{e}&=&\lambda\,\rho\, f_{0}+\lambda\, g_{0}-\lambda\,\rho\,
b,\quad\quad\quad\quad\quad\quad\quad\quad\quad\quad\quad\quad\quad\quad\quad\quad\quad\quad\quad
\end{array}
\right.
\]
where $b:G\rightarrow\mathbb{C}$ is a bounded function,
$\rho\in\mathbb{C},\,\lambda\in\mathbb{C}\setminus\{0\}$ are
constants and $f_{0},g_{0}:G\rightarrow\mathbb{C}$ are functions
satisfying the cosine functional equation
$$f_{0}(x+y)=f_{0}(x)f_{0}(y)-g_{0}(x)g_{0}(y),\,\,x,y\in G.$$
\par Notice that $f_{0}\not\in\mathcal{B}(G)$ because
$f^{e}=-\lambda^{2}\,f_{0}+\lambda^{2}\,b$,
$f^{e}\not\in\mathcal{B}(G)$ and $b\in\mathcal{B}(G)$. Since $f^{e}$
and $h^{e}$ are linearly independent modulo $\mathcal{B}(G)$ so are
$f_{0}$ and $g_{0}$. Indeed, if not then there exist a constant
$\alpha\in\mathbb{C}$ and a function $\varphi\in\mathcal{B}(G)$ such
that $g_{0}=\alpha f_{0}+\varphi$. Hence $h^{e}=\lambda\,\rho\,
f_{0}+\lambda\,(\alpha f_{0}+\varphi)-\lambda\,\rho\,
b=\lambda\,(\rho+\alpha)f_{0}+b_{1}$, where
$b_{1}:=\lambda\,\varphi-\lambda\,\rho\,b$ belongs to
$\mathcal{B}(G)$. Then
$\lambda\,h^{e}+(\rho\,+\alpha)f^{e}=\lambda\,b_{1}+\lambda^{2}(\rho\,+\alpha)b$,
which implies that the function $\lambda\, h^{e}+(\rho+\alpha)f^{e}$
belongs to $\mathcal{B}(G)$. This contradicts the fact that $f^{e}$
and $h^{e}$ are linearly independent modulo $\mathcal{B}(G)$ because
$\lambda\,\neq0$. Hence $f_{0}$ and $g_{0}$ are linearly independent
modulo $\mathcal{B}(G)$.
\par On the other hand let $\psi_{1}:=f^{o}$, $\psi_{2}:=g^{o}$ and $\psi_{3}:=h^{o}$.
The identity (\ref{eq3}) implies
\begin{equation*}\begin{split}&\varphi_{1}(x,y)=(-\lambda^{2}f_{0}(x)+\lambda^{2}b(x))\psi_{2}(y)
+(\frac{1+\rho^{2}}{2}f_{0}(x)+\rho\, g_{0}(x)+\frac{1-\rho^{2}}{2}b(x))\psi_{1}(y)\\
&+(\lambda\,\rho\, f_{0}(x)+\lambda\, g_{0}(x)-\lambda\,\rho\, b(x))\psi_{3}(y)\\
&=f_{0}(x)[-\lambda^{2}\psi_{2}(y)+\frac{1+\rho^{2}}{2}\psi_{1}(y)+\lambda\,\rho\,\psi_{3}(y)]+g_{0}(x)[\rho\,\psi_{1}(y)+\lambda\,\psi_{3}(y)]\\
&+b(x)[\lambda^{2}\psi_{2}(y)+\frac{1-\rho^{2}}{2}\psi_{1}(y)-\lambda\,\rho\,\psi_{3}(y)],\end{split}\end{equation*}
for all $x,y\in G$. So, taking (\ref{eq13}) into account and that
the functions $\psi$, $b$, $\psi_{1}$, $\psi_{2}$ and $\psi_{3}$ are
bounded, we deduce from the identity above that the function
$$x\mapsto f_{0}(x)[-\lambda^{2}\,\psi_{2}(y)+\frac{1+\rho^{2}}{2}\psi_{1}(y)+\lambda\,\rho\,\psi_{3}(y)]+g_{0}(x)[\rho\,\psi_{1}(y)+\lambda\,\psi_{3}(y)]$$
belongs to $\mathcal{B}(G)$ for all $y\in G$. As $f_{0}$ and $g_{0}$
are linearly independent modulo $\mathcal{B}(G)$ we get that
$$-\lambda^{2}\psi_{2}(y)+\frac{1+\rho^{2}}{2}\psi_{1}(y)+\lambda\,\rho\,\psi_{3}(y)=0$$
and
$$\rho\,\psi_{1}(y)+\lambda\,\psi_{3}(y)=0$$
for all $y\in G$, from which we get by a small computation that
$\psi_{2}=\frac{1-\rho^{2}}{2\lambda^{2}}\psi_{1}$ and
$\psi_{3}=-\frac{\rho\,}{\lambda\,}\psi_{1}$. As
$f=f^{e}+f^{o}=f^{e}+\psi_{1}$,
$g=g^{e}+g^{o}=g^{e}+\psi_{2}=g^{e}+\frac{1-\rho^{2}}{2\lambda^{2}}\psi_{1}$
and
$h=h^{e}+h^{o}=h^{e}+\psi_{3}=h^{e}+\frac{1-\rho^{2}}{2\lambda^{2}}\psi_{1}$,
we deduce that
\begin{center}
\[(I) \left\{
\begin{array}{r c l}
f&=&-\lambda^{2}\,f_{0}+\lambda^{2}\,b+\psi_{1}\\
g&=&\frac{1+\rho^{2}}{2}\,f_{0}+\rho\,g_{0}+\frac{1-\rho^{2}}{2}\,b+\frac{1-\rho^{2}}{2\lambda^{2}}\psi_{1}\\
h&=&\lambda\,\rho\,f_{0}+\lambda\,g_{0}-\lambda\,\rho\,b-\frac{\rho\,}{\lambda\,}\psi_{1}
\end{array}
\right.
\]
\end{center}
Moreover, since $f^{e}$, $g^{e}$ and $h^{e}$ are even functions, and $\psi_{1}=f^{o}$, we get that\\
\[ \left\{
\begin{array}{r c l}
\psi_{1}(-x)&=&-\psi_{1}(x)\quad\quad\quad\quad\quad\quad\quad\quad\quad\quad\quad\quad\quad\quad\quad\quad\quad\quad\quad\quad\quad\quad\\
-f_{0}(-x)+b(-x)&=&-f_{0}(x)+b(x)\quad\quad\quad\quad\quad\quad\quad\quad\quad\quad\quad\quad\quad\quad\quad\quad\quad\quad\quad\quad\quad\quad\\
\frac{1}{2}f_{0}(-x)+\rho\,g_{0}(-x)+\frac{1}{2}b(-x)&=&\frac{1}{2}f_{0}(x)+\rho\,g_{0}(x)+\frac{1}{2}b(x)\quad\quad\quad\quad\quad\quad\quad\quad\quad\quad\quad\quad\quad\quad\quad\\
\rho\,f_{0}(-x)+g_{0}(-x)-\rho\,b(-x)&=&\rho\,f_{0}(x)+g_{0}(x)-\rho\,b(x),\quad\quad\quad\quad\quad\quad\quad\quad\quad\quad\quad\quad\quad\quad\quad\quad\quad\quad\quad
\end{array}
\right.
\]
which implies $f_{0}(-x)=f_{0}(x)$, $g_{0}(-x)=g_{0}(x)$,
$b(-x)=b(x)$ and $\psi_{1}(-x)=-\psi_{1}(x)$ for all $x\in G$. So we
obtain, by writing $b$ instead of
$b+\frac{1}{\lambda^{2}}\,\psi_{1}$ in $(I)$, the result (5) of
Theorem 4.1.\\
(2)\[ \left\{
\begin{array}{r c l}
f^{e}&=&\lambda^{2}\,M+a\,m+b,\quad\quad\quad\quad\quad\quad\quad\quad\quad\quad\quad\quad\quad\quad\\
g^{e}&=&\beta\lambda\,(1-\frac{1}{2}\beta\lambda\,)M+(1-\beta\lambda\,)m-\frac{1}{2}\beta^{2}\,a\,m-\frac{1}{2}\beta^{2}\,b,\quad\quad\quad\quad\quad\quad\quad\\
h^{e}&=&\lambda\,(1-\beta\lambda\,)M-\lambda\,m-\beta\,a\,m-\beta\,b,\quad\quad\quad\quad\quad\quad\quad\quad\quad\quad\quad\quad
\end{array}
\right.
\]
where $m:G\rightarrow\mathbb{C}$ is a nonzero bounded multiplicative
function, $M:G\rightarrow\mathbb{C}$ is a non bounded multiplicative
function, $a:G\rightarrow\mathbb{C}$ is a nonzero additive function,
$b:G\rightarrow\mathbb{C}$ is a bounded function and
$\beta\in\mathbb{C},\,\lambda\in\mathbb{C}\setminus\{0\}$ are
constants. Then
$\beta\,f^{e}+h^{e}=\beta\lambda^{2}\,M+\beta\,a\,m+\beta\,b+\lambda\,(1-\beta\lambda\,)M-\lambda\,m-\beta\,a\,m-\beta\,b=\lambda\,(M-m)$.\\
So that
\begin{equation}\label{eq29}M(-x)-m(-x)=M(x)-m(x)\end{equation}
for all $x\in G$. Moreover, since $f^{e}$ and $g^{e}$ are even
functions, and $a(-x)+a(x)=a(-x+x)=a(e)=0$ for all $x\in G$, we get
that
\begin{equation}\label{eq30}\lambda^{2}\,M(-x)-a(x)\,m(-x)+b(-x)=\lambda^{2}\,M(x)+a(x)\,m(x)+b(x)\end{equation}
and
\begin{equation}\label{eq31}\begin{split}\beta\lambda\,(1-\frac{1}{2}\beta\lambda\,)M(-x)+(1-\beta\lambda\,)m(-x)+\frac{1}{2}\beta^{2}\,a(x)\,m(-x)-\frac{1}{2}\beta^{2}\,b(-x)\\
=\beta\lambda\,(1-\frac{1}{2}\beta\lambda\,)M(x)+(1-\beta\lambda\,)m(x)-\frac{1}{2}\beta^{2}\,a(x)\,m(x)-\frac{1}{2}\beta^{2}\,b(x),\end{split}\end{equation}
for all $x\in G$. By multiplying (\ref{eq30}) by
$\frac{1}{2}\beta^{2}$ and adding the result to (\ref{eq31}) we get
that
$$\beta\lambda\,(M(x)-m(x))-\beta\lambda\,(M(-x)-m(-x))+m(x)-m(-x)=0$$ for all $x\in
G$. We deduce, by taking (\ref{eq29}) into account, that
$m(-x)=m(x)$ and $M(-x)=M(x)$ for all $x\in G$. When we substitute
this back into (\ref{eq30}) we get that
$$-a(x)\,m(x)+b(-x)=a(x)\,m(x)+b(x)$$ for all $x\in G$. Hence
$a(x)=-b^{o}(x)\,m(-x)$ for all $x\in G$. As $b$ and $m$ are bounded
functions we derive that the additive function $a$ is bounded, so
$a(x)=0$ for all $x\in G$, which contradicts the condition on $a$.
So the present possibility dose not occur.\\
(3)\[ \left\{
\begin{array}{r c l}
f^{e}&=&\frac{1}{2}a^{2}\,m+\frac{1}{2}a_{1}\,m+b,\quad\quad\quad\quad\quad\quad\quad\quad\quad\quad\quad\quad\quad\quad\quad\\
g^{e}&=&-\frac{1}{4}\beta^{2}\,a^{2}\,m+\beta\,a\,m-\frac{1}{4}\beta^{2}\,a_{1}\,m+m-\frac{1}{2}\beta^{2}\,b,\quad\quad\quad\quad\quad\quad\quad\quad\\
h^{e}&=&-\frac{1}{2}\beta\,a^{2}\,m+a\,m-\frac{1}{2}\beta\,a_{1}\,m-\beta\,b,\quad\quad\quad\quad\quad\quad\quad\quad\quad\quad\quad\quad\quad
\end{array}
\right.
\]
where $m:G\rightarrow\mathbb{C}$ is a nonzero bounded multiplicative
function, $a,a_{1}:G\rightarrow\mathbb{C}$ are additive functions
such that $a$ is nonzero, $b:G\rightarrow\mathbb{C}$ is a bounded
function and $\beta\in\mathbb{C}$ is a constant.\\
Notice that $\beta\,f^{e}+h^{e}=a\,m$ and
$2g^{e}=\beta^{2}\,f^{e}+2\beta\,h^{e}+2m$, then $m$ and $a\,m$ are
even functions. As seen earlier we have $a(-x)=-a(x)$ for all $x\in
G$. Hence $-a(x)\,m(x)=a(x)\,m(x)$ for all $x\in G$, so $a=0$, which
contradicts the condition on $a$. We conclude that the present
possibility dose not occur.\\
(4)
$$f^{e}(x+y)=f^{e}(x)m(y)+m(x)f^{e}(y)+(a(x)m(x)+b(x))(a(y)m(y)+b(y))$$for
all $x,y\in G$,
$$g^{e}=-\frac{1}{2}\beta^{2}\,f^{e}+(1+\beta\,a)m+\beta\,b$$ and
$$h^{e}=-\beta\,f^{e}+a\,m+b,$$
where $m:G\rightarrow\mathbb{C}$ is a nonzero bounded multiplicative
function, $a:G\rightarrow\mathbb{C}$ is a nonzero additive function,
$b:G\rightarrow\mathbb{C}$ is a bounded function and
$\beta\in\mathbb{C}$ is a constant.
\par The second and the third identities above imply
$m=-\frac{1}{2}\beta^{2}\,f^{e}+g^{e}-\beta\,h^{e}$, from which we
deduce that $m(-x)=m(x)$ for all $x\in G$. Moreover the third
identity above implies that the function $a\,m+b$ is even. Since
$a(-x)=-a(x)$ for all $x\in G$, we get that
$-a(x)m(x)+b(-x)=a(x)m(x)+b(x)$ for all $x\in G$. Hence
$a=-b^{o}\,m$. As $b$ and $m$ are bounded functions and $a$ is an
additive function we deduce that $a=0$, which contradicts the
condition on $a$. We conclude that the present possibility dose not
occur.\\
(5) $f^{e}$, $g^{e}$ and $h^{e}$ satisfy the functional equation
\begin{equation}\label{eq32}f^{e}(x+y)=f^{e}(x)g^{e}(y)+g^{e}(x)f^{e}(y)+h^{e}(x)h^{e}(y)\end{equation}
for all $x,y\in G$.
\par If $f^{o}=0$ then $f^{e}=f$. Moreover, taking into account that $f^{e}$ and $h^{e}$
are linearly independent, we derive from (\ref{eq3}) that $g^{o}=0$
and $h^{o}=0$, hence $g^{e}=g$ and $h^{e}=h$. So the functional
equation (\ref{eq32}) becomes $f(x-y)=f(x)g(y)+g(x)f(y)+h(x)h(y)$
for all $x,y\in G$. This is the result (8) of Theorem 4.1.
\par If $f^{o}\neq0$ then, according to (\ref{eq3}), there exist two constants
$\alpha,\beta\in\mathbb{C}$ and an even function
$b\in\mathcal{B}(G)$ such that
\begin{equation}\label{eq33}g^{e}=\alpha\,f^{e}+\beta\,h^{e}+b.\end{equation}
By substituting (\ref{eq33}) into (\ref{eq32}) we get, by a similar
computation to the one of Case A of the proof of \cite[Lemma
3.4]{Ajebbar and Elqorachi3}, that
\begin{equation}\label{eq34}\begin{split}f^{e}(x+y)&=(2\alpha-\beta^{2})f^{e}(x)f^{e}(y)+f^{e}(x)b(y)+b(x)f^{e}(y)\\
&+[\beta\,f^{e}(x)+h^{e}(x)][\beta\,f^{e}(y)+h^{e}(y)]\end{split}\end{equation}
for all $x,y\in G$. We have the following subcases:\\
\underline{Subcase B.1.1}: $2\alpha\neq\beta^{2}$. Proceeding
exactly as in Subcase A.1 of the proof of \cite[Lemma 3.4]{Ajebbar
and Elqorachi3} we get that
\[ \left\{
\begin{array}{r c l}
f^{e}&=&-\lambda^{2}f_{0}+\lambda^{2}b,\quad\quad\quad\quad\quad\quad\quad\quad\quad\quad\quad\quad\quad\quad\quad\quad\quad\quad\quad\quad\quad\quad\\
g^{e}&=&\frac{1+\rho^{2}}{2}f_{0}+\rho\, g_{0}+\frac{1-\rho^{2}}{2}b,\quad\quad\quad\quad\quad\quad\quad\quad\quad\quad\quad\quad\quad\quad\quad\\
h^{e}&=&\lambda\,\rho\, f_{0}+\lambda\, g_{0}-\lambda\,\rho\,
b.\quad\quad\quad\quad\quad\quad\quad\quad\quad\quad\quad\quad\quad\quad\quad\quad\quad\quad\quad
\end{array}
\right.
\]
So we go back to the possibility (1) and then obtain the result (5)
of Theorem 4.1.\\
\underline{Subcase B.1.1}: $2\alpha=\beta^{2}$.\\
By similar computations to the ones in Subcase A.1 of the proof of
\cite[Lemma 3.4]{Ajebbar and Elqorachi3} we get that there exist a
constant $\eta\in\mathbb{C}$ such that
\begin{equation}\label{eq35}H(x+y)=H(x)m(y)+m(x)H(y)+\eta\,H(x)H(y)\end{equation}
for all $x,y\in G$\\ and
\begin{equation}\label{eq36}b=m\end{equation} where $\eta\in\mathbb{C}$, $H:=\beta\,f^{e}+h^{e}$ and
$m\in\mathcal{B}(G)$ is an even multiplicative function.
\par If $\eta=0$ then $H$ satisfies the functional equation
$$H(x+y)=H(x)m(y)+m(x)H(y)$$ for all $x,y\in G$. As $f^{e}$ and
$h^{e}$ are linearly independent modulo $\mathcal{B}(G)$ we have
$H\neq0$, hence $m$ is a nonzero multiplicative function on the
group $G$. So, from the functional equation above we deduce that
there exists an additive function $a:G\rightarrow\mathbb{C}$  such
that $H=a\,m$. Since $H$ is even so is $a$, hence $a=0$ which
contradicts the fact that $H\neq0$.
\par If $\eta\neq0$ then, by multiplying both sides of (\ref{eq35})
by $\eta$ and adding $m(x+y)$ to both sides of the obtained
identity, we get, by a small computation, that
\begin{equation*}m(x+y)+\eta^{2}H(x+y)=[m(x)+\eta\,H(x)][m(y)+\eta\,H(y)]\end{equation*}
for all $x,y\in G$. So there exist an even multiplicative function
$M:G\rightarrow\mathbb{C}$ and a constant
$\lambda\in\mathbb{C}\backslash\{0\}$ such that $H=\lambda(M-m)$. By
substituting this into (\ref{eq34}) and taking (\ref{eq36}) into
account we obtain
\begin{equation*}\begin{split}f^{e}(x+y)&=f^{e}(x)m(y)+m(x)f^{e}(y)+\lambda^{2}(M(x)-m(x))(M(y)-m(y))\\
&=f^{e}(x)m(y)+m(x)f^{e}(y)+\lambda^{2}M(x+y)-\lambda^{2}M(x)m(y)-\lambda^{2}m(x)M(y)\\&+\lambda^{2}m(x+y)\end{split}\end{equation*}
for all $x,y\in G$. Since $m$ is a nonzero multiplicative function
on the group $G$ we have $m(x)\neq0$ for all $x\in G$. So, by
dividing both sides of the functional equation above we get that
$$\dfrac{f^{e}(x+y)-\lambda^{2}M(x+y)}{m(x+y)}+\lambda^{2}=[\dfrac{f^{e}(x)-\lambda^{2}M(x)}{m(x)}+\lambda^{2}]+[\dfrac{f^{e}(y)-\lambda^{2}M(y)}{m(y)}+\lambda^{2}]$$
for all $x,y\in G$, hence there exists an additive function
$a:G\rightarrow\mathbb{C}$ such that
$\dfrac{f^{e}(x)-\lambda^{2}M(x)}{m(x)}+\lambda^{2}=a(x)$ for all
$x\in G$. Since $f^{e}$, $M$ and $m$ are even functions so is the
additive function $a$, then $a(x)=0$ for all $x\in G$.
Hence $f^{e}=\lambda^{2}(M-m)$. Then
$f^{e}=\lambda\,H=\lambda\beta\,f^{e}+\lambda\,h^{e}$, which
contradicts the linear independence modulo $\mathcal{B}(G)$ of
$f^{e}$ and $h^{e}$. We conclude that the Subcase B.1.1 does not occur. \\
\underline{Subcase B.2}: $h^{o}\not\in\mathcal{B}(G)$. Since
$\mathcal{B}(G)$ is a two-sided invariant and $(-I)$-invariant
linear space of complex-valued functions on $G$, then we deduce,
according to Lemma 3.2, that $h^{e}=\gamma\,f^{e}$ and
$g^{o}=-\gamma\,h^{o}-\eta\,f^{o}$, where $\gamma,\eta\in\mathbb{C}$
are two constants. We split the discussion into the cases $\gamma=0$
and $\gamma\neq0$.\\
\underline{Subcase B.2.1}: $\gamma=0$. Then, from Lemma 3.1(1),
(\ref{eq17}) and (\ref{eq18}), we deduce that $h^{o}=h$ and
$g^{o}\in\mathcal{B}(G)$. So we get, from the identities (\ref{eq5})
and (\ref{eq6}), that
\begin{equation*}
\begin{split}f(x+y)&=f(x)g(y)+g(x)f(y)+h(x)h(y)-2f^{o}(x)g^{o}(y)-2g^{o}(x)f^{o}(y)\\
&-2h(x)h(y)+\psi(x,y)+\varphi_{3}(x,y)\\
&=[f^{e}(x)+f^{o}(x)][g^{e}(y)+g^{o}(y)]+[g^{e}(x)+g^{o}(x)][f^{e}(y)+f^{o}(y)]\\
&-h(x)h(y)-2f^{o}(x)g^{o}(y)-2g^{o}(x)f^{o}(y)+\psi(x,y)+\varphi_{3}(x,y)\\
&=f^{e}(x)g^{e}(y)+g^{e}(x)f^{e}(y)-h(x)h(y)+(f^{e}(x)g^{o}(y)+g^{e}(x)f^{o}(y))\\
&+(g^{o}(x)f^{e}(y)+f^{o}(x)g^{e}(y))-f^{o}(x)g^{o}(y)-g^{o}(x)f^{o}(y)+\psi(x,y)\\
&+\varphi_{3}(x,y)\end{split}\end{equation*} for all $x,y\in G$.
Hence, taking into account that $h^{e}=0$, and by using (\ref{eq3})
and (\ref{eq16}), a small computation shows that
\begin{equation}\label{eq37}f^{e}(x+y)=f^{e}(x)g^{e}(y)+g^{e}(x)f^{e}(y)+k(x)k(y)+\Psi(x,y)\end{equation}
for all $x,y\in G$, where
\begin{equation}\label{eq38}k:=ih\end{equation} and
\begin{equation}\label{eq39}\Psi(x,y):=\psi(x,-y)+\varphi_{1}(y,x)-\varphi_{1}(x,y)-f^{o}(x+y)-f^{o}(x)g^{o}(y)-g^{o}(x)f^{o}(y)\end{equation}
for all $x,y\in G$. As the functions $f^{o},\,g^{o}$ and $\psi$ are
bounded we deduce, from (\ref{eq13}), (\ref{eq37}) and (\ref{eq39}),
that the function $$(x,y)\mapsto
f^{e}(x+y)-f^{e}(x)g^{e}(y)-g^{e}(x)f^{e}(y)-k(x)k(y)$$ is bounded.
Hence, according to Proposition 3.4 we have one of the following
possibilities:\\
(1)\[ \left\{
\begin{array}{r c l}
f^{e}&=&-\lambda^{2}\,f_{0}+\lambda^{2}\,b,\quad\quad\quad\quad\quad\quad\quad\quad\quad\quad\quad\quad\quad\quad\quad\quad\quad\quad\quad\quad\quad\quad\\
g^{e}&=&\frac{1+\rho^{2}}{2}\,f_{0}+\rho\,g_{0}+\frac{1-\rho^{2}}{2}\,b,\quad\quad\quad\quad\quad\quad\quad\quad\quad\quad\quad\quad\quad\quad\quad\\
k&=&\lambda\,\rho\,f_{0}+\lambda\,g_{0}-\lambda\,\rho\,b,\quad\quad\quad\quad\quad\quad\quad\quad\quad\quad\quad\quad\quad\quad\quad\quad\quad\quad\quad
\end{array}
\right.
\]
where $b:G\rightarrow\mathbb{C}$ is a bounded function,
$\rho\in\mathbb{C},\,\lambda\in\mathbb{C}\setminus\{0\}$ are
constants and $f_{0},g_{0}:G\rightarrow\mathbb{C}$ are functions
satisfying the cosine functional equation
$$f_{0}(x+y)=f_{0}(x)f_{0}(y)-g_{0}(x)g_{0}(y),\,\,x,y\in G.$$
\par Since $f^{e}$ and $g^{e}$ are even functions, $k$ is an odd
function and $\lambda\,\neq0$ we get that
\begin{equation}\label{eq40}f_{0}(-x)-b(-x)=f_{0}(x)-b(x),\end{equation}
\begin{equation}\label{eq41}f_{0}(-x)+2\rho\,g_{0}(-x)+b(-x)=f_{0}(x)+2\rho\,g_{0}(x)+b(x)\end{equation}
and
\begin{equation}\label{eq42} \rho\,(f_{0}(-x)-b(-x))+g_{0}(-x)=-\rho\,(f_{0}(x)-b(x))-g_{0}(x)\end{equation}
for all $x\in G$. The identity (\ref{eq40}) implies
\begin{equation}\label{eq43} f_{0}^{o}=b^{o}\end{equation}
By using this and the identity
$k=\lambda\,\rho\,f_{0}+\lambda\,g_{0}-\lambda\,\rho\,b$, and taking
into account that $k$ is an odd function we obtain
\begin{equation}\label{eq44} k=\lambda\,g_{0}^{o}.\end{equation}
By multiplying both sides of (\ref{eq40}) by $\rho$ and subtracting
(\ref{eq42}) from the result we deduce that
\begin{equation}\label{eq45} g_{0}^{e}=-\rho\,(f_{0}-b).\end{equation}
Moreover, we derive from (\ref{eq41}) that
$2\rho\,(g_{0}(x)-g_{0}(-x))=-(f_{0}(x)-f_{0}(-x))-(b(x)-b(-x))$ for
all $x\in G$, which implies, by taking (\ref{eq43}) into account,
that
\begin{equation}\label{eq46} \rho\,g_{0}^{o}=-b^{o}.\end{equation}
From (\ref{eq44}), (\ref{eq46}) and (\ref{eq38}) we get that
\begin{equation}\label{eq47} \rho\,h=\lambda\,ib^{o}.\end{equation}
Since $b$ is a bounded function on $G$ we deduce from (\ref{eq47})
that $\rho\,h$ is. As $h\not\in\mathcal{B}(G)$ we get that $\rho=0$.
It follows that
\begin{center}
\[(II) \left\{
\begin{array}{r c l}
f^{e}&=&-\lambda^{2}\,f_{0}+\lambda^{2}\,b,\\
g^{e}&=&\frac{1}{2}f_{0}+\frac{1}{2}\,b,\\
k&=&\lambda\,g_{0}.
\end{array}
\right.
\]
\end{center}
\par Let $\psi_{1}:=g^{o}$ and $\psi_{2}:=f^{o}$. By using that
$h^{e}=0$, (\ref{eq3}), the first and the second identities in
$(II)$ we obtain
\begin{equation*}\begin{split}&\varphi_{1}(x,y)=(-\lambda^{2}\,f_{0}(x)+\lambda^{2}\,b(x))\psi_{1}(y)+(\frac{1}{2}\,f_{o}(x)+\frac{1}{2}\,b(x))\psi_{2}(y)\\
&=f_{0}(x)[-\lambda^{2}\,\psi_{1}(y)+\frac{1}{2}\psi_{2}(y)]+b(x)[\lambda^{2}\,\psi_{1}(y)+\frac{1}{2}\psi_{2}(y)],\end{split}\end{equation*}
for all $x,y\in G$. So, taking (\ref{eq13}) into account and that
the functions $\psi$, $b$, $\psi_{1}$ and $\psi_{2}$ are bounded, we
deduce from the identity above that the function $$x\mapsto
f_{0}(x)[-\lambda^{2}\,\psi_{1}(y)+\frac{1}{2}\psi_{2}(y)]$$ belongs
to $\mathcal{B}(G)$ for all $y\in G$. Since
$f^{e}=-\lambda^{2}\,f_{0}+\lambda^{2}\,b$,
$f^{e}\not\in\mathcal{B}(G)$ and $b\in\mathcal{B}(G)$ we deduce that
$f_{0}\not\in\mathcal{B}(G)$. Hence
$-\lambda^{2}\,\psi_{1}(y)+\frac{1}{2}\psi_{2}(y)=0$ for all $y\in
G$, which implies $\psi_{2}=2\lambda^{2}\,\psi_{1}$. As
$f=f^{e}+f^{o}=f^{e}+\psi_{2}=f^{e}+2\lambda^{2}\,\psi_{1}$,
$g^{e}+g^{o}=g^{e}+\psi_{1}$ we deduce, taking (\ref{eq38}) and
$(II)$ into account, that
\begin{center}
\[(III)\left\{
\begin{array}{r c l}
f&=&-\lambda^{2}\,f_{0}+\lambda^{2}\,b+2\lambda^{2}\,\psi_{1},\\
g&=&\frac{1}{2}f_{0}+\frac{1}{2}\,b+\psi_{1},\\
h&=&-\lambda i\,g_{0}.
\end{array}
\right.
\]
\end{center}
On the other hand, we get from the identities (\ref{eq46}),
(\ref{eq43}), (\ref{eq45}) and $\psi_{1}=g^{o}$, that $b(-x)=b(x)$,
$f_{0}(-x)=f_{0}(x)$, $g_{0}(-x)=-g_{0}(x)$ and
$\psi_{1}(-x)=-\psi_{1}(x)$ for all $x\in G$, and
$\psi_{1}\in\mathcal{B}(G)$. So we obtain, by writing $b$ and
$\lambda$ instead of $b+2\psi_{1}$ and $-\lambda i$ respectively in
$(III)$, the result (6) of Theorem 4.1.\\
(2)\[ \left\{
\begin{array}{r c l}
f^{e}&=&\lambda^{2}\,M+am+b,\quad\quad\quad\quad\quad\quad\quad\quad\quad\quad\quad\quad\quad\quad\\
g^{e}&=&\beta\lambda(1-\frac{1}{2}\beta\lambda)M+(1-\beta\lambda)m-\frac{1}{2}\beta^{2}\,a\,m-\frac{1}{2}\beta^{2}\,b,\quad\quad\quad\quad\quad\quad\quad\\
k&=&\lambda(1-\beta\lambda)M-\lambda\,m-\beta\,a\,m-\beta\,b,\quad\quad\quad\quad\quad\quad\quad\quad\quad\quad\quad\quad
\end{array}
\right.
\]
where $m:G\rightarrow\mathbb{C}$ is a nonzero bounded multiplicative
function, $M:G\rightarrow\mathbb{C}$ is a non bounded multiplicative
function, $a:G\rightarrow\mathbb{C}$ is a nonzero additive function,
$b:G\rightarrow\mathbb{C}$ is a bounded function and
$\beta\in\mathbb{C},\,\lambda\in\mathbb{C}\setminus\{0\}$ are
constants.
\par We have $\beta\,k=-\frac{1}{2}\beta^{2}\,f^{e}+g^{e}-m$,
which implies, taking into account that $k$ is an odd function, that
$\beta\,k=-m^{o}$. Hence $\beta\,k\in\mathcal{B}(G)$. As
$k\not\in\mathcal{B}(G)$ we get that $\beta=0$. Then $g^{e}=m$ and
$k=\lambda(M-m)$. Since $\lambda\neq0$ we get that $m(-x)=m(x)$ and
$M(-x)-m(-x)=-M(x)+m(x)$ for all $x\in G$. So that
$2m(x)=M(-x)+M(x)$ for all $x\in G$. Since $m$ and $M$ are
multiplicative functions we deduce, according to \cite[Corollary
3.19]{Stetkaer}, that $m=M$, which contradicts the conditions
$m\in\mathcal{B}(G)$ and $M\not\in\mathcal{B}(G)$. Thus
the present possibility does not occur.\\
(3)\[ \left\{
\begin{array}{r c l}
f^{e}&=&\frac{1}{2}a^{2}\,m+\frac{1}{2}a_{1}\,m+b,\quad\quad\quad\quad\quad\quad\quad\quad\quad\quad\quad\quad\quad\quad\quad\\
g^{e}&=&-\frac{1}{4}\beta^{2}\,a^{2}\,m+\beta\,a\,m-\frac{1}{4}\beta^{2}\,a_{1}\,m+m-\frac{1}{2}\beta^{2}\,b,\quad\quad\quad\quad\quad\quad\quad\quad\\
k&=&-\frac{1}{2}\beta\,a^{2}\,m+a\,m-\frac{1}{2}\beta\,a_{1}\,m-\beta\,b,\quad\quad\quad\quad\quad\quad\quad\quad\quad\quad\quad\quad\quad
\end{array}
\right.
\]
where $m:G\rightarrow\mathbb{C}$ is a nonzero bounded multiplicative
function, $a,a_{1}:G\rightarrow\mathbb{C}$ are additive functions
such that $a$ is nonzero, $b:G\rightarrow\mathbb{C}$ is a bounded
function and $\beta\in\mathbb{C}$ is a constant.
\par Notice that $\beta\,k=-\frac{1}{2}\beta^{2}\,f^{e}+g^{e}-m$. As in the
possibility above we get that $\beta=0$. Hence we obtain
\begin{center}
\[(IV) \left\{
\begin{array}{r c l}
f^{e}&=&\frac{1}{2}a^{2}\,m+\frac{1}{2}a_{1}\,m+b,\\
g^{e}&=&m,\\
k&=&a\,m.
\end{array}
\right.
\]
\end{center}
From the second identity of $(IV)$ we deduce that $m(-x)=m(x)$ for
all $x\in G$. As $f^{e}(-x)=f^{e}(x)$, $a(-x)=-a(x)$ and
$a_{1}(-x)=-a_{1}(x)$ for all $x\in G$, we deduce from the first
identity of $(IV)$ that
$\frac{1}{2}a^{2}(x)m(x)-\frac{1}{2}a_{1}(x)m(x)+b(-x)=\frac{1}{2}a^{2}(x)m(x)+\frac{1}{2}a_{1}(x)m(x)+b(x)$
for all $x\in G$. So $a_{1}(x)m(x)=b(x)-b(-x)$ for all $x\in G$,
from which we get, taking into account that $m(-x)=m(x)$ for all
$x\in G$ and $m$ is a nonzero multiplicative function on the group
$G$, that $a_{1}=-2mb^{o}$. As $m,b\in\mathcal{B}(G)$ and $a_{1}$ is
an additive function we deduce that $a_{1}=0$ and $b(-x)=b(x)$ for
all $x\in G$. Hence the first identity of $(IV)$ becomes
$f^{e}=\frac{1}{2}a^{2}\,m+b$. So, taking into account that
$g^{e}=m$ and $h^{e}=0$, the identity (\ref{eq3}) becomes
\begin{equation*}\begin{split}&\varphi_{1}(x,y)=[\frac{1}{2}a^{2}(x)m(x)+b(x)]g^{o}(y)+m(x)f^{o}(y)\\
&=\frac{1}{2}a^{2}(x)m(x)g^{o}(y)+b(x)g^{o}(y)+m(x)f^{o}(y),\end{split}\end{equation*}
for all $x,y\in G$. As the functions $m$, $b$, $g^{o}$ and $f^{o}$
are bounded and $m$ is a nonzero multiplicative function on the
group $G$, we deduce from the identity above that the function
$$x\mapsto a^{2}(x)g^{o}(y)$$ belongs to $\mathcal{B}(G)$ for all $y\in
G$. Since $a^{2}$ is a non bounded function, because $a$ is a
nonzero additive function on $G$, we deduce that $g^{o}=0$. We infer
from $(IV)$, taking (\ref{eq38}) into account, and using that
$f=f^{e}+f^{o}$ and $g=g^{e}+g^{o}$, that
\begin{center}
\[\left\{
\begin{array}{r c l}
f&=&\frac{1}{2}a^{2}\,m+b+f^{o},\\
g&=&m,\\
h&=&-ia\,m.
\end{array}
\right.
\]
\end{center}
By writing $b$ instead of $b+f^{o}$ in the identities above we
obtain the result $(7)$ of Theorem 4.1.\\
(4) $f^{e}$ satisfies the functional equation
\begin{equation}\label{eq48}f^{e}(x+y)=f^{e}(x)m(y)+m(x)f^{e}(y)+(a(x)m(x)+b(x))(a(y)m(y)+b(y))\end{equation}
for all $x,y\in G$,
$$g^{e}=-\frac{1}{2}\beta^{2}\,f^{e}+(1+\beta\,a)m+\beta\,b$$ and
$$k=-\beta\,f^{e}+a\,m+b,$$
where $m:G\rightarrow\mathbb{C}$ is a nonzero bounded multiplicative
function, $a:G\rightarrow\mathbb{C}$ is a nonzero additive function,
$b:G\rightarrow\mathbb{C}$ is a bounded function and
$\beta\in\mathbb{C}$ is a constant.
\par A small computation shows that
$\beta\,k=-\frac{1}{2}\beta^{2}\,f^{e}+g^{e}-m$. So, as in the
possibility (2), we have $\beta=0$. Hence
\begin{equation}\label{eq49}g^{e}=m\end{equation}
and
\begin{equation}\label{eq50}k=a\,m+b.\end{equation}
From (\ref{eq48}) and (\ref{eq50}) we deduce that $f^{e}$ and $k$
satisfy the functional equation
$$f^{e}(x+y)=f^{e}(x)m(y)+m(x)f^{e}(y)+k(x)k(y).$$
As $a$ is a nonzero additive function, $m$ is a nonzero
multiplicative bounded function and $b$ is bounded we derive from
(\ref{eq50}) that $k\neq0$. Moreover $k(-x)=-k(x)$ for all $x\in G$,
and from (\ref{eq49}) we get that $m(-x)=m(x)$ for all $x\in G$.
Hence, according to Proposition 3.3, $f^{e}$ and $k$ are the forms
\begin{equation}\label{eq51}f^{e}=\frac{1}{2}A^{2}m\end{equation}
and
\begin{equation}\label{eq52}k=A\,m,\end{equation}
where $A:G\rightarrow\mathbb{C}$ is a nonzero additive function. It
follows, from (\ref{eq50}), (\ref{eq52}) and that $m(-x)=m(x)$ for
all $x\in G$, that $A-a=b\,m$. Hence, $A-a$ is a bounded additive
function. So $A=a$ and $b=0$. We deduce, taking (\ref{eq51}) and
(\ref{eq52}) into account, that
\begin{equation}\label{eq53}f^{e}=\frac{1}{2}a^{2}m.\end{equation}
and
\begin{equation}\label{eq54}k=a\,m.\end{equation}
Moreover, since the functions are $m$ and $\psi$ bounded, we deduce
by using (\ref{eq3}), (\ref{eq13}) and (\ref{eq49}), that the
function $x\rightarrow f^{e}(x)g^{o}(y)$ belongs to $\mathcal{B}(G)$
for all $y\in G$. As seen earlier, we have
$f^{e}\not\in\mathcal{B}(G)$. Hence
\begin{equation}\label{eq55}g^{o}=0.\end{equation}
So, by using (\ref{eq38}), (\ref{eq49}), (\ref{eq53}), (\ref{eq54})
and (\ref{eq55}), and taking into account that
$f^{o}\in\mathcal{B}(G)$, we conclude, by writing $b$ instead of
$f^{o}$, that
\begin{center}
\[ \left\{
\begin{array}{r c l}
f&=&\frac{1}{2}a^{2}\,m+b,\\
g&=&m,\\
h&=&-ia\,m.
\end{array}
\right.
\]
\end{center}
The result occurs in (7) of Theorem 4.1.\\
(5) $f^{e}$, $g^{e}$ and $k$ satisfy the functional equation
\begin{equation}\label{eq56}f^{e}(x+y)=f^{e}(x)g^{e}(y)+g^{e}(x)f^{e}(y)+k(x)k(y)\end{equation}
for all $x,y\in G$.
\par If $f^{o}=0$ then $f^{e}=f$. Moreover we derive from (\ref{eq18}) that $g^{e}=g$.
So, by using (\ref{eq38}), the functional equation (\ref{eq56})
becomes $f(x+y)=f(x)g(y)+g(x)f(y)-h(x)h(y)$ for all $x,y\in G$. As
$h=h^{o}$ we derive that $f$, $g$ and $h$ satisfy the functional
equation $f(x-y)=f(x)g(y)+g(x)f(y)+h(x)h(y)$ for all $x,y\in G$.
This is the result (8) of Theorem 4.1.
\par If $f^{o}\neq0$ then, according to (\ref{eq3}), there exist a
constant $\eta\in\mathbb{C}$ and an even function
$\varphi\in\mathcal{B}(G)$ such that
\begin{equation*}g^{e}=\eta\,f^{e}+\varphi.\end{equation*}
Substituting this into (\ref{eq56}) we obtain
\begin{equation}\label{eq57}f^{e}(x+y)=2\eta\,f^{e}(x)f^{e}(y)+f^{e}(x)\varphi(y)+\varphi(x)f^{e}(y)+k(x)k(y)\end{equation}
for all $x,y\in G$.
\par If $\eta=0$, then the functional equation (\ref{eq57}) can
be written
\begin{equation}\label{eq58}f^{e}(x+y)=f^{e}(x)\varphi(y)+\varphi(x)f^{e}(y)+k(x)k(y)\end{equation}
for all $x,y\in G$. Notice that $\varphi\neq0$. Indeed, if
$\varphi=0$ then we get, by putting $y=e$ in (\ref{eq58}) and taking
(\ref{eq38}) into account, that $f^{e}(x)+h(x)h(e)=0$ for all $x\in
G$. Since $h=h^{o}$ we have $h(e)=0$. Hence $f^{e}(x)=0$ for all
$x\in G$, and then $f=f^{o}$, which implies $f\in\mathcal{B}(G)$ and
contradicts that $f$ and $h$ are linearly independent modulo
$\mathcal{B}(G)$. Moreover we derive from (\ref{eq58}), according to
\cite[Lemma 3.2]{Ajebbar and Elqorachi3}, that $\varphi$ is a
multiplicative function because $f^{e}$ and $k$ are linearly
independent modulo $\mathcal{B}(G)$ and $\varphi\in\mathcal{B}(G)$.
Let $m:=\varphi$. Then the functional equation (\ref{eq58}) becomes
$$f^{e}(x+y)=f^{e}(x)m(y)+m(x)f^{e}(y)+k(x)k(y)$$ for all $x,y\in
G$. Since $f^{e}$ is an even function, $m$ a nonzero multiplicative
function on the group $G$ such that
$m(-x)=\varphi(-x)=\varphi(x)=m(x)$ for all $x\in G$, and $k$ an odd
function we deduce, according to Proposition 3.3, that
$f^{e}=\frac{1}{2}a^{2}m$ and $k=a\,m$ where
$a:G\rightarrow\mathbb{C}$ is a nonzero additive function. So,
taking (\ref{eq38}), (\ref{eq18}) and (\ref{eq19}) into account, and
using that $f^{o}\in\mathcal{B}(G)$, $\gamma=\eta=0$ and
$\varphi=m$, we derive, by putting $b=f^{o}$, that
\begin{center}
\[ \left\{
\begin{array}{r c l}
f&=&\frac{1}{2}a^{2}\,m+b,\\
g&=&m,\\
h&=&-ia\,m.
\end{array}
\right.
\]
\end{center}
This is the result (7) of Theorem 4.1.
\par If $\eta\neq0$, let $\lambda\in\mathbb{C}\setminus\{0\}$ such that $\lambda^{2}=\frac{1}{2\eta}$. The functional equation (\ref{eq57})
can be written, by multiplying both sides by $\frac{1}{\lambda^{2}}$
and adding $\varphi(x+y)$ to the obtained functional equation, as
follows
\begin{equation*}\begin{split}&\frac{1}{\lambda^{2}}\,f^{e}(x+y)+\varphi(x+y)=
[\frac{1}{\lambda^{2}}\,f^{e}(x)+\varphi(x)][\frac{1}{\lambda^{2}}\,f^{e}(y)+\varphi(y)]+\frac{1}{\lambda^{2}}\,k(x)k(y)\\
&+\varphi(x+y)-\varphi(x)\varphi(y)\end{split}\end{equation*} for
all $x,y\in G$. As $\varphi\in\mathcal{B}(G)$ we get that the
function
$$x\mapsto\frac{1}{\lambda^{2}}\,f^{e}(x+y)+\varphi(x+y)-[\frac{1}{\lambda^{2}}\,f^{e}(x)+\varphi(x)][\frac{1}{\lambda^{2}}\,f^{e}(y)+\varphi(y)]-\frac{1}{\lambda^{2}}\,k(x)k(y)$$
belongs to the two-sided invariant linear space $\mathcal{B}(G)$ for
all $y\in G$. Since the functions $f^{e}$ and $h$ are linearly
independent modulo $\mathcal{B}(G)$ so are
$\frac{1}{\lambda^{2}}\,f^{e}+\varphi$ and
$\frac{1}{\lambda^{2}}\,k$. Hence, according to \cite[Lemma
3.1]{Székelyhidi} and taking (\ref{eq38}) into account, the
functional equation
\begin{equation*}\frac{1}{\lambda^{2}}\,f^{e}(x+y)+\varphi(x+y)=
[\frac{1}{\lambda^{2}}\,f^{e}(x)+\varphi(x)][\frac{1}{\lambda^{2}}\,f^{e}(y)+\varphi(y)]-\frac{1}{\lambda^{2}}h(x)h(y)\end{equation*}
for all $x,y\in G$, is satisfied, from which we deduce that
\begin{center}
\[(V) \left\{
\begin{array}{r c l}
f^{e}&=&\lambda^{2}\,f_{0}-\lambda^{2}\varphi,\\
h&=&\lambda\,g_{0},
\end{array}
\right.
\]
\end{center}
where $f_{0}:=\frac{1}{\lambda^{2}}\,f^{e}+\varphi$ and
$g_{0}:=\frac{1}{\lambda\,}h$ satisfy the functional equation
$$f_{0}(x+y)=f_{0}(x)f_{0}(y)-g_{0}(x)g_{0}(y)$$ for all $x,y\in G$.
Moreover, since $\varphi$ is an even function and $h^{e}=0$ we get
easily that $f_{0}(-x)=f_{0}(x)$ and $g_{0}(-x)=-g_{0}(x)$ for all
$x\in G$. On the other hand, by taking into account that
$f=f^{e}+f^{o}$ and $g=g^{e}+g^{o}$, and by using (\ref{eq18}),
(\ref{eq19}) and $(V)$, we derive by an elementary computation that
\begin{center}
\[\left\{
\begin{array}{r c l}
f&=&\lambda^{2}\,f_{0}-\lambda^{2}\,b,\\
g&=&\frac{1}{2}\,f_{0}+\frac{1}{2}\,b\\
h&=&\lambda\,g_{0},
\end{array}
\right.
\]
\end{center}
where $b:=\varphi-\frac{1}{\lambda^{2}}\,f^{o}$ is a bounded
function. The result occurs in (6) of Theorem 4.1.\\
\underline{Subcase B.2.2}: $\gamma\neq0$. Let $x,y\in G$ be
arbitrary. By substituting (\ref{eq17}) and (\ref{eq18}) in
(\ref{eq3}) we obtain by an elementary computation
\begin{equation}\label{eq59}\varphi_{1}(x,y)=[-\eta\, f^{e}(x)+g^{e}(x)]f^{o}(y).\end{equation}
On the other hand, since $f=f^{e}+f^{o}$ and $g=g^{e}+g^{o}$ the
identity (\ref{eq6}) can be written
\begin{equation*}\begin{split}&\psi(x,y)=f^{e}(x-y)-f^{e}(x)g^{e}(y)-g^{e}(x)f^{e}(y)-g^{e}(x)f^{o}(y)-f^{e}(x)g^{o}(y)\\
&-f^{e}(y)g^{o}(x)-f^{o}(x)g^{e}(y)-f^{o}(x)g^{o}(y)-g^{o}(x)f^{o}(y)-h(x)h(y)+f^{o}(x-y).\end{split}\end{equation*}
By using (\ref{eq18}) we obtain
\begin{equation*}\begin{split}&\psi(x,y)=f^{e}(x-y)-f^{e}(x)g^{e}(y)-g^{e}(x)f^{e}(y)-h(x)h(y)-g^{e}(x)f^{o}(y)\\
&-f^{e}(x)[-\gamma\, h^{o}(y)-\eta\, f^{o}(y)]-f^{e}(y)[-\gamma\, h^{o}(x)-\eta\, f^{o}(x)]-f^{o}(x)g^{e}(y)\\
&-f^{o}(x)[-\gamma\, h^{o}(y)-\eta\, f^{o}(y)]-f^{o}(y)[-\gamma\, h^{o}(x)-\eta\, f^{o}(x)]+f^{o}(x-y)\\
&=f^{e}(x-y)-f^{e}(x)g^{e}(y)-g^{e}(x)f^{e}(y)-h(x)h(y)+\gamma\, f^{e}(x)h^{o}(y)+\gamma\, f^{e}(y)h^{o}(x)\\
&+\gamma\, f^{o}(x)h^{o}(y)+\gamma\, h^{o}(x)f^{o}(y)+2\eta\, f^{o}(x)f^{o}(y)-[-\eta\, f^{e}(x)+g^{e}(x)]f^{o}(y)\\
&-[-\eta\,
f^{e}(y)+g^{e}(y)]f^{o}(x)+f^{o}(x-y),\end{split}\end{equation*}
from which we infer, by using that $h=h^{e}+h^{o}$, and taking
(\ref{eq17}) and (\ref{eq59}) account, that
\begin{equation*}\begin{split}&\psi(x,y)=f^{e}(x-y)-f^{e}(x)g^{e}(y)-g^{e}(x)f^{e}(y)-[h^{e}(x)+h^{o}(x)][h^{e}(y)+h^{o}(y)]\\
&+h^{e}(x)h^{o}(y)+h^{e}(y)h^{o}(x)+\gamma\, f^{o}(x)h^{o}(y)+\gamma\, h^{o}(x)f^{o}(y)-\varphi_{1}(x,y)-\varphi_{1}(y,x)\\
&+2\eta\,f^{o}(x)f^{o}(y)+f^{o}(x-y)\\
&=f^{e}(x-y)-f^{e}(x)g^{e}(y)-g^{e}(x)f^{e}(y)-h^{e}(x)h^{e}(y)-h^{o}(x)h^{o}(y)+\gamma\, f^{o}(x)h^{o}(y)\\
&+\gamma\, h^{o}(x)f^{o}(y)-\varphi_{1}(x,y)-\varphi_{1}(y,x)+2\eta\, f^{o}(x)f^{o}(y)+f^{o}(x-y)\\
&=f^{e}(x-y)-f^{e}(x)g^{e}(y)-g^{e}(x)f^{e}(y)-\gamma^{2}f^{e}(x)f^{e}(y)-h^{o}(x)h^{o}(y)\\
&+\gamma\,f^{o}(x)h^{o}(y)\ +\gamma\,
h^{o}(x)f^{o}(y)-\varphi_{1}(x,y)-\varphi_{1}(y,x)+2\eta\,
f^{o}(x)f^{o}(y)+f^{o}(x-y).\end{split}\end{equation*} So that
\begin{equation}\label{eq60}\begin{split}&f^{e}(x-y)-f^{e}(x)[g^{e}(y)+\frac{1}{2}\gamma^{2}f^{e}(y)]-[g^{e}(x)+\frac{1}{2}\gamma^{2}f^{e}(x)]f^{e}(y)\\
&-[h^{o}(x)-\gamma\, f^{o}(x)][h^{o}(y)-\gamma\, f^{o}(y)]\\
&=\psi(x,y)+\varphi_{1}(x,y)+\varphi_{1}(y,x)-(\gamma^{2}+2\eta)f^{o}(x)f^{o}(y)-f^{o}(x-y)\end{split}\end{equation}
for all $x,y\in G$. Let
\begin{equation}\label{eq61}F_{0}:=f^{e},\,G_{0}:=g^{e}+\frac{1}{2}\gamma^{2}f^{e},\,H_{0}:=h^{o}-\gamma\,f^{o}.\end{equation}
Since $f=f^{e}+f^{o}$, $g=g^{e}+g^{o}$ and $h=h^{e}+h^{o}$, we get
by putting $\delta=-\gamma$ and $\varphi=f^{o}$, and taking
(\ref{eq17}), (\ref{eq18}) and (\ref{eq61}) into account, that
\begin{center}
\[(VI)\left\{
\begin{array}{r c l}
f&=&F_{0}+\varphi,\\
g&=&-\frac{1}{2}\delta^{2}F_{0}+G_{0}+\delta\,H_{0}-(\eta+\delta^{2})\varphi,\\
h&=&-\delta\,F_{0}+H_{0}-\delta\,\varphi.
\end{array}
\right.
\]
\end{center}
If $\varphi=0$ the result (9) of Theorem 4.1 is obviously satisfied.
In the following we assume that $\varphi\neq0$. By using
(\ref{eq59}), the first identity and the second one in (\ref{eq60}),
and replacing $f^{o}$ by $\varphi$, we get, by a small computation,
that
\begin{equation*}\varphi_{1}(x,y)=-[(\eta+\frac{1}{2}\delta^{2})F_{0}(x)-G_{0}(x)]\varphi(y)\end{equation*}
for all $x,y\in G$. Since $f^{o}$ and $\psi$ are bounded functions,
we deduce, taking (\ref{eq13}) and the identity above into account,
that
\begin{equation}\label{eq62}(\eta+\frac{1}{2}\delta^{2})F_{0}-G_{0}\in\mathcal{B}(G),\end{equation}
and, from (\ref{eq60}) and (\ref{eq61}), we derive that the function
$$(x,y)\mapsto F_{0}(x-y)-F_{0}(x)G_{0}(y)-G_{0}(x)F_{0}(y)-H_{0}(x)H_{0}(y)$$
is bounded. Since $f$ and $h$ are linearly independent modulo
$\mathcal{B}(G)$, we deduce easily, by using the first and the third
identities in (\ref{eq60}), that $H_{0}$ and $F_{0}$ are because
$f^{o}\in\mathcal{B}(G)$ and $h^{o}\not\in\mathcal{B}(G)$. Moreover
we have $H_{0}^{o}=H_{0}$ and $H_{0}^{o}\not\in\mathcal{B}(G)$,
hence we go back to Subcase B.2.1. As $F_{0}$ and $G_{0}$ are
even functions we derive that we have the following subcases:\\
\underline{Subcase B.2.2.1:} $F_{0},\,G_{0},$ and $H_{0}$ are of the
form (6) with the same constraints. Then
$$F_{0}=\lambda^{2}f_{0}-\lambda^{2}b,\,G_{0}=\frac{1}{2}f_{o}+\frac{1}{2}b,\,H_{0}=\lambda\,g_{0},$$
where $\lambda\in\mathbb{C}\setminus\{0\}$ is a constant and
$b,f_{0},g_{0}:G\rightarrow\mathbb{C}$ are functions satisfying the
same constraints indicated in (6) of Theorem 4.1, unless to take
$b(-x)=b(x)$ for all $x\in G$, then a small computation shows, by
using (\ref{eq62}) and the formulas of $F_{0}$ and $G_{0}$, that
$[\frac{1}{2}-\lambda^{2}(\eta+\frac{1}{2}\delta^{2})]f_{0}\in\mathcal{B}(G)$.
As $F_{0}$ and $H_{0}$ are linearly independent modulo
$\mathcal{B}(G)$ and $b\in\mathcal{B}(G)$, we get
$f_{0}\not\in\mathcal{B}(G)$. So that
$\frac{1}{2}-\lambda^{2}(\eta+\frac{1}{2}\delta^{2})=0$ and then
$\eta=\frac{1}{2\lambda^{2}}-\frac{1}{2}\delta^{2}$. By substituting
this back into $(VI)$ we obtain the result (9) of Theorem 4.1 with
the constraint (i).\\
\underline{Subcase B.2.2.2:} $F_{0},\,G_{0},$ and $H_{0}$ are of the
form (7) with the same constraints. Then we get, taking into account
that $F_{0}(-x)=F_{0}(x)$ and $b(-x)=-b(x)$ for all $x\in G$, that
$b=0$. So that
$$F_{0}=\frac{1}{2}a^{2}\,m,\,G_{0}=m,\,H_{0}=-ia\,m,$$
where $m:G\rightarrow\mathbb{C}$ is a nonzero bounded multiplicative
function, $a:G\rightarrow\mathbb{C}$ is a nonzero additive function
such that $m(-x)=m(x)$ for all $x\in G$. By using (\ref{eq62}) and
the formulas of $F_{0}$ and $G_{0}$ we get, by an elementary
computation, that
$(\eta+\frac{1}{2}\delta^{2})a^{2}\in\mathcal{B}(G)$. Since $a$ is a
nonzero additive function we get that $a^{2}\not\in\mathcal{B}(G)$.
Hence $\eta=-\frac{1}{2}\delta^{2}$. By substituting this back into
$(VI)$ we obtain the result (9) of
Theorem 4.1 with the constraint (ii).\\
\underline{Subcase B.2.2.3:} $F_{0},\,G_{0},$ and $H_{0}$ satisfy
the functional equation in the result (8) of Theorem 4.1, i.e.,
$F_{0}(x-y)=F_{0}(x)G_{0}(y)+G_{0}(x)F_{0}(y)+H_{0}(x)H_{0}(y)$ for
all $x,y\in G$. Since $F_{0}$ and $G_{0}$ are even functions and
$H_{0}$, replacing $y$ by $-y$ yields to the functional equation
$$F_{0}(x+y)=F_{0}(x)G_{0}(y)+G_{0}(x)F_{0}(y)+(iH_{0}(x))(iH_{0}(y)).$$
From (\ref{eq62}) we derive that there exist a constant
$\alpha\in\mathbb{C}$ and a function $b_{0}\in\mathcal{B}(G)$ such
that $G_{0}=\frac{\alpha}{2}F_{0}+b_{0}$. So that the last
functional equation becomes
\begin{equation*}F_{0}(x+y)=\alpha\,F_{0}(x)F_{0}(y)+F_{0}(x)b_{0}(y)+b_{0}(x)F_{0}(y)+(iH_{0}(x))(iH_{0}(y)),\end{equation*}
for all $x,y\in G$. Hence, by using a similar idea used to solve
(\ref{eq57}) (see Subcase B.2.1(5)) we prove that:
\par If $\alpha=0$, then $F_{0}=\frac{1}{2}a^{2}\,m$, $G_{0}=m$ and $H_{0}=-ia\,m$,
where $m:G\rightarrow\mathbb{C}$ is a nonzero bounded multiplicative
function such that $m(-x)=m(x)$ for all $x\in G$, so we go back to
Subcase B.2.2.2 and obtain the result (9) of Theorem 4.1 with the
constraint (ii).
\par If $\alpha\neq0$, then
$F_{0}=\lambda^{2}\,f_{0}-\lambda^{2}\,b_{0}$,
$G_{0}=\frac{1}{2}\,f_{0}+\frac{1}{2}\,b$ and
$H_{0}=\lambda\,g_{0}$, where $b:G\rightarrow\mathbb{C}$ is a
bounded function, $\lambda\in\mathbb{C}\setminus\{0\}$ is a constant
and $f_{0},g_{0}:G\rightarrow\mathbb{C}$ are functions satisfying
the cosine functional equation
$f_{0}(x+y)=f_{0}(x)f_{0}(y)-g_{0}(x)g_{0}(y)$ for all $x,y\in G$,
such that $f_{0}(-x)=f_{0}(x)$, $g_{0}(-x)=-g_{0}(x)$ and
$b(-x)=-b(x) $ for all $x\in G$, so we go back to Subcase B.2.2.1
and obtain the result (9) of Theorem 4.1 with the constraint (i).
\par Conversely if $f,g$ and $h$ are of the
forms (1)-(9) in Theorem 4.1 we check by elementary computations
that the function $(x,y)\mapsto f(x-y)-f(x)g(y)-g(x)f(y)-h(x)h(y)$
is bounded. This completes the proof.
\end{proof}
%---------------------------------------------------------------%

\end{document}